\documentclass[11pt,reqno]{amsart}
\usepackage{fullpage}
\usepackage{xcolor}
\usepackage[T1]{fontenc}                                   
\usepackage{amsfonts}
\usepackage[utf8]{inputenc}                               
\usepackage{comment}                                       
\usepackage{mparhack}                                      
\usepackage{amsmath,amssymb,amsthm,mathrsfs,eucal}                      
\usepackage{booktabs}                                                                          
\usepackage{graphicx,subfig}                                                                
\usepackage{wrapfig}                                                                            
\usepackage[bookmarks=true,colorlinks=true]{hyperref}                      
\usepackage{bm}                                                                                   

\usepackage{dirtytalk}

\usepackage{cancel}
\usepackage{soul}

\usepackage{dsfont}

\newtheorem{defn}{Definition}[section]
\newtheorem{thm}{Theorem}[section]
\newtheorem{prop}{Proposition}[section]
\newtheorem{lem}{Lemma}[section]
\newtheorem{cor}{Corollary}[section]

\newtheorem{rem}{Remark}[section]

\DeclareMathOperator*{\argmin}{argmin}
\newcommand{\mF}{\mathcal{F}}
\newcommand{\mP}{{\mathcal{P}}}

\newcommand{\mptr}{{\mathcal{P}_2(\R)}}

\newcommand{\mptra}{{\mathcal{P}_2^a(\R)}}

\newcommand{\mW}{\mathcal{W}}

\newcommand{\R}{\mathbb{R}}

\newcommand{\RN}{{\mathbb{R}^{N}}}

\newcommand{\supp}{\mathrm{supp}}
\newcommand{\sign}{\mathrm{sign}}

\def\avint{\mathop{\,\,\rlap{\bf{--}}\!\!\int}\nolimits}

\def\XXint#1#2#3{{\setbox0=\hbox{$#1{#2#3}{\int}$}
  \vcenter{\hbox{$#2#3$}}\kern-.5\wd0}}

\allowdisplaybreaks

\usepackage{soul}

\title{Many-particle  limit for a system of interaction equations driven by Newtonian potentials}
\date{}

\begin{document}
\author{M. Di Francesco \and A. Esposito \and M. Schmidtchen}
\address{M. Di Francesco - DISIM - Department of Information Engineering, Computer Science and Mathematics, University of L'Aquila, Via Vetoio 1 (Coppito) 67100 L'Aquila (AQ) - Italy}
\address{A. Esposito -- 
Department Mathematik, Friedrich-Alexander-Universit\"at Erlangen-N\"urnberg, Cauerstrasse 11,
91058 Erlangen, Germany.}
\address{M. Schmidtchen --  Sorbonne Universit\'e, Universit\'e de Paris, Laboratoire Jacques-Louis Lions, F-75005, Paris, France.}
\email{marco.difrancesco@univaq.it}
\email{antonio.esposito@fau.de}
\email{schmidtchen@ljll.maths.upmc.fr}

\begin{abstract}
    We consider a discrete particle system of two species coupled through nonlocal interactions driven by the one-dimensional Newtonian potential, with repulsive self-interaction and attractive cross-interaction. After providing a suitable existence theory in a finite-dimensional framework, we explore the behaviour of the particle system in case of collisions and analyse the behaviour of the solutions with initial data featuring particle clusters. Subsequently, we prove that the empirical measure associated to the particle system converges to the unique 2-Wasserstein gradient flow solution of a system of two partial differential equations (PDEs) with nonlocal interaction terms in a proper measure sense. The latter result uses uniform estimates of the $L^m$-norms of a piecewise constant reconstruction of the density using the particle trajectories.
\end{abstract}

\keywords{systems with Newtonian interactions, deterministic particle limit, mean-field limit, gradient flows}
\subjclass[2010]{Primary: 35A24, 35F55. Secondary: 82C22; 35A35; 35Q70; 35R09}







\maketitle

\section{Introduction}
The general problem of approximating transport PDEs by the empirical measure associated to moving particles is quite classical in many contexts such as particle physics and gravitation. We refer to cornerstone papers such as \cite{morrey1955,onsager1944,dobrushin,sznitman} and to the review paper \cite{golse}. A prototype model and variation of the pure transport PDE  which gained great attention in the last decades is the nonlocal transport-diffusion equation
\begin{equation}\label{example}
    \partial_t \rho = \mathrm{div}(\nabla a(\rho) + \rho \nabla G\ast \rho),
\end{equation}
where $a=a(\rho)$ is a nonlinear diffusion function and $G$ is a space dependent kernel modelling nonlocal interaction. In the aforementioned contexts in particle physics and gravitation, $G$ is typically a singular kernel, which makes the analysis of \eqref{example} quite challenging. A similar situation occurs in the study of Keller-Segel model for chemotaxis, more precisely in its parabolic-elliptic version, see e.g. \cite{jager,biler,bertozzi,blanchet}. A different situation arises e.g. in \cite{BCP97,CMcCV03,toscani_granular} in the analysis of mean-field models for granular media, in which $G$ is typically a power law of the form $G(x)=|x|^\alpha$ with $\alpha>1$.

In the context of modern applications and real-world problems, equations of the form \eqref{example} naturally arise in the description of aggregation phenomena in population dynamics, see
\cite{boi,MEK99,okubo,topaz}. In these works the nonlocal terms are coupled with a linear or nonlinear diffusion arising from stochastic noise, see \cite{masgallic,oelschlaeger}. Clearly, the classical results in \cite{StroockVaradhan79,GuoPapaniVar88} are also relevant in this context although less related from the methodological point of view.

\bigskip
Starting from the early 2000 years, the theory of gradient flows in Wasserstein spaces developed in \cite{otto,JKO98,AGS} became an important tool to provide well-posedness results for the class of models \eqref{example}. \cite{otto} first provided the seminal ideas leading to the formulation of the porous medium equation as a gradient flow in the Wasserstein sense, whereas \cite{JKO98} adapted the 
``minimising movement'' idea by De Giorgi to the new metric framework. The case with nonlocal interactions was first studied in \cite{CMcCV03}, which partly anticipated the results in \cite{AGS} without providing the full metric framework but adapting the theory to the case of \eqref{example}, whereas \cite{AGS} addresses a more general theory of gradient flows in metric spaces. The result in \cite{CDFFLS} is also relevant in this context in that it allowed to extend the theory to kernels $G$ displaying a discontinuity of the gradient at the origin in the re-solution of the JKO-scheme \cite{JKO98} and in the proof of $\lambda$-convexity \cite{McC97} of the related functional. Moreover, \cite{CDFFLS} also provides a finite-time blow-up result for solutions with data in the space of probability measures. 

The role of $\lambda$-convexity of the functional $\mathcal{F}:\mptr\rightarrow \R$ defined by
\[\mu\mapsto \mathcal{F}[\mu]=\int_\R G\ast \mu\, d\mu\]
(here $\mptr$ denotes the space of probability measures with finite second moment) is essential in order to prove a stability result for two solution curves $\mu(t), \nu(t)$ (leading also to uniqueness of measure solutions) of the form
\[W_2(\mu(t),\nu(t))\leq e^{-\lambda t}W_2(\mu(0),\nu(0)),\]
which often implies as a byproduct  a \emph{many-particle approximation} result for the target equation \eqref{example}. This is true both in the diffusion-free case and in the case with diffusion, see the recent \cite{carrillo_delgadino_pavliotis}. 

The situation is more complicated in cases in which the functional lacks the $\lambda$-convexity, which is typically the case when $G$ has a singularity at the origin.
Attractive singularities make the study of well-posedness quite challenging in $L^p$ spaces. In this context, the result in \cite{bertozzi} allows to prove existence and uniqueness up to the blow-up time or globally when the singularity is not too strong. The repulsive case is also challenging especially if one wants to prove a many-particle approximation result, because a strong repulsion at the origin for $G$ forces point particles to resolve into absolutely continuous measures. A quite thorough study of the many-particle approximation in the absence of diffusion and with 
``almost Newtonian'' singular kernels (both attractive and repulsive) was provided in \cite{CCH14}, based on a technique developed in \cite{HaurayJabin07}.  In the case of $a=0$, the result in \cite{BCDFP} provides, so far, the only many-particle approximation result for \eqref{example} with $G$ being the Newtonian potential, albeit in one space dimension (i.e. $G(x)=\pm |x|$). Such a result also explores the connection of \eqref{example} with a scalar conservation law satisfied by the cumulative distribution variable $\int^x \rho(y,t) dy$.

The diffusion-free case $a=0$ often allows for a significant ``reduction of complexity'' of the PDE under consideration in that it often permits to approximate it by a set of \emph{deterministic} particles, i.e. not subject to stochastic noise and simply obeying a system of ordinary differential equations. Obvious advantages of that are the possibility or approximating the density under consideration by a discrete set of Lagrangian trajectories (a feature of great impact in some applications such as traffic flow or pedestrian movements) and the availability of a new 
numerical ``particle'' method for the target PDE. In fact, recent contributions to the literature try to provide deterministic approximations to transport PDEs in the case with diffusion as well, see the classical \cite{Russo90} for one dimensional linear diffusion, the result in \cite{GosseToscani06} for one-dimensional nonlinear diffusion, the results in \cite{patacchini} for multidimensional diffusion.

\bigskip
Recently, the specialised literature displayed an increasing interest of \emph{systems} of gradient flows, i.e. systems of more than one transport equations of the form \eqref{example}, modelling the mutual interplay of more than one species of individuals. The case with diffusion has a very rich literature in that it is quite challenging at the level of well-posedness due to the possibility of cross-diffusion effects. We refer for instance to the recent \cite{DiFranEspFag}, which provides a general existence result of two-species gradient flows of functionals with cross-diffusion and nonlocal interactions terms. 

As in the one-species case, the well-posedness and the stability in a Wasserstein gradient flow sense are strictly related with the convergence of a deterministic particle approximation scheme. We mention in this context the result in \cite{DFF} which allows to prove singular behavior such as a total collapse of particles and cluster formations via stability in the Wasserstein gradient flow sense of \cite{AGS}. The general (diffusion-free) system considered in \cite{DFF} reads
\begin{equation}\label{eq:pde-sys}
\begin{cases}
\partial_t\rho=\partial_x(\rho H_1'*\rho)+\partial_x(\rho K_1'*\eta),\\
\partial_t\eta=\partial_x(\eta H_2'*\eta)+\partial_x(\eta K_2'*\rho),
\end{cases}
\end{equation}
where the given potentials $H_1, H_2, K_1, K_2$ are smooth enough and convex up to a quadratic perturbation.

The recent result in \cite{CDFEFS} extended the existence and uniqueness proven in \cite{DFF} to the one-dimensional Newtonian case 
\begin{equation}\label{eq:newtonian_potentials}
    H_1(x)=H_2(x)=-|x|\, , \qquad K_1(x)=K_2(x)=|x|,
\end{equation}
corresponding to a set of particles of two species, with mutual repulsion within the same species (self-repulsion, or intra-specific repulsion) and attraction between particles of opposite species (cross-attraction, or inter-specific attraction), the driving interaction kernels being multiples of the Newtonian potential. The result of \cite{CDFEFS} holds in one space dimension. In particular, in case of absolutely continuous initial data $\rho_0, \eta_0$ \cite{CDFEFS} proves  global-in-time existence and uniqueness of solutions by posing system \eqref{eq:pde-sys}-\eqref{eq:newtonian_potentials} as gradient flow of the \textit{interaction energy functional}
\begin{equation}\label{eq:int_en_functional}
\mF(\rho,\eta)=-\frac{1}{2}\int_\R N*\rho\,d\rho - \frac{1}{2}\int_\R N*\eta\,d\eta + \int_\R N*\eta\,d\rho,
\end{equation}
where
\[N(x):=|x|, \qquad x\in\R.\] 
When dealing with general measures as initial data, in particular Dirac deltas, the sub-differential of $\mF$ may be empty (see \cite{BCDFP}). Hence, in \cite{CDFEFS}, global-in-time existence and uniqueness of solutions to system \eqref{eq:pde-sys}-\eqref{eq:newtonian_potentials} is proven by (formally) re-writing the system in the pseudo-inverse formalism and by using the concept of gradient flows in Hilbert spaces \textit{\`a la Br\'ezis}, cf. \cite{Brezis}.
With potentials $H_1=H_2=-K_1=-K_2$ featuring a repulsive singularity of logarithmic type at the origin, system \eqref{eq:pde-sys} has also been studied in the context of multi-sign systems (arising e.g. in semiconductor theory) and evolution models for dislocations in crystals, cf.  \cite{mainini2012,alicandro2014metastability,garroni20}. 

\bigskip
In this paper we prove that the PDE system \eqref{eq:pde-sys}-\eqref{eq:newtonian_potentials}, namely
\begin{equation}\label{eq:pde-sys-newtonian}
\begin{cases}
\partial_t\rho=-\partial_x(\rho \partial_x |\cdot|*\rho)+\partial_x(\rho \partial_x |\cdot|*\eta),\\
\partial_t\eta=-\partial_x(\eta \partial_x|\cdot|*\eta)+\partial_x(\eta \partial_x |\cdot|*\rho),
\end{cases}
\end{equation}
can be obtained as the \emph{many-particle limit} of the deterministic ODE system
\begin{equation}\label{eq:ode-sys-newtonian}
\begin{cases}
\displaystyle{\dot x_i(t)=\sum_{x_k(t)\neq x_i(t)}m_k\sign(x_i(t)-x_k(t))-\sum_{y_k(t)\neq x_i(t)}n_k\sign(x_i(t)-y_k(t))},\\
\displaystyle{\dot y_j(t)=\sum_{x_k(t)\neq x_i(t)}n_k\sign(y_j(t)-y_k(t))-\sum_{y_k(t)\neq x_i(t)}m_k\sign(y_j(t)-x_k(t))},
\end{cases}
\end{equation}
with $i=1,..,N$, and $j=1,...,N$. System \eqref{eq:ode-sys-newtonian} models the movement of $N$ particle for each species, with masses $m_1,\ldots,m_N$ for the $x$-species and $n_1,\ldots,n_N$ for the $y$-species, under the effect of repulsive Newtonian potentials for same-species interactions and attractive Newtonian potentials for cross-species interactions.

We stress that, unlike the associated scalar model studied in \cite{BCDFP}, particles in the ODE system \eqref{eq:ode-sys-newtonian} \emph{may overlap}. When this happens, the right-hand side of \eqref{eq:ode-sys-newtonian} features a jump discontinuity, which brings additional difficulties. To bypass this problem and to better understand the dynamics of \eqref{eq:ode-sys-newtonian}, we frame it rigorously as the (finite dimensional) gradient flow of the (convex, in a suitable metric sense) functional
\[-\frac{1}{2}\sum_{i,j}m_i m_j|x_i-x_j| - \frac{1}{2}\sum_{i,j}n_i n_j|y_i-y_j| + \sum_{i,j}m_i n_j |x_i-y_j|,\]
in the convex cone $\mathcal{C}^N\times \mathcal{C}^N$ of ordered configurations 
\[x_1\leq x_2\leq \ldots\leq x_N\,,\qquad y_1\leq y_2\leq \ldots \leq y_N.\]
More precisely, among other issues:
\begin{itemize}
    \item We prove that the sub-differential of this functional is always non-empty for any given configuration in $\mathcal{C}^N\times \mathcal{C}^N$ (including overlapping of particles of opposite species).
    \item We analyse \emph{collisions} among particles (which are possible because particles do not ``slow down'' when they get very close due to the lack of regularity of the interaction potential) and prove that \emph{particles of the same species never collide}. Moreover, we provide explicit necessary and sufficient conditions for particles of opposite species to cross each other.
    \item We explore the case of initial overlapping of particles and provide the explicit solution to the corresponding particle system.
\end{itemize}
These properties are preparatory to prove the main result of this paper, which is the rigorous derivation of  solutions to  \eqref{eq:pde-sys-newtonian} with $L^1$ initial data as many-particle limits of the empirical measures of the particle system \eqref{eq:ode-sys-newtonian}. 

More precisely, we consider a pair of nonnegative initial densities $\rho_0, \eta_0 \in L^1$, both with unit mass. We approximate them via atomic measures by considering 
$N$ particles, $x_1,x_2,...,x_N$, of the first species and another $N$ particles, $y_1,...,y_N$, of the second species, with non-zero masses, $m_1,...,m_N$ and $n_1,...,n_N$ respectively, such that $\sum_{i=1}^Nm_i=\sum_{k=1}^Nn_k=1$. We let those particles evolve according to the ODE system \eqref{eq:ode-sys-newtonian}. We then prove that the empirical measures
\begin{equation*}
    \rho^N(t,x)=\frac{1}{N}\sum_{i=1}^{N}\delta_{x_i(t)}(x), \quad \text{and}\quad \eta^N(t,x)=\frac{1}{N}\sum_{j=1}^{N}\delta_{y_j(t)}(x),
\end{equation*}
converge to the unique  gradient flow solution to \eqref{eq:pde-sys-newtonian} in a suitable distributional sense as $N\rightarrow+\infty$.

Such a result heavily relies on uniform estimates at the discrete level. In particular, we observe that in the $2$-species case  weak compactness in the measure sense by itself is insufficient to obtain consistency in the limit due to the cross-interaction terms. Indeed, since the cross-interaction terms cannot be symmetrised (unlike, for instance, the Keller-Segel one-species model), \emph{weak $L^1$ compactness} is needed in this case. We shall explain this issue in detail in Section \ref{sec:many_particle_limit}.

We emphasise that system \eqref{eq:pde-sys-newtonian} is not included in the theory of \cite{DFF} since the interaction potential in the (repulsive) intraspecific parts of $\mF$ is neither convex nor $\lambda$-convex, i.e., convex up to a quadratic perturbation. In one dimension this problem can be overcome as shown in \cite{CDFEFS}. Another difference with \cite{DFF} is that the analysis in \cite{CDFEFS} implies that \emph{particle solutions} are not  \emph{gradient flow solutions} to system \eqref{eq:pde-sys-newtonian}. Thus, the mean-field limit cannot be treated via the stability result mentioned previously since the atoms of the empirical measure may diffuse instantaneously.

Finally, we observe that our result is of interest in the framework posed in  \cite{garroni20} for two-species models for dislocations with logarithmic singular potentials, more precisely \eqref{eq:pde-sys} with $H_1=H_2=-K_1=-K_2$ having a singularity at the origin \say{not stronger} than a logarithmic one. In \cite{garroni20}, the convergence of a discrete particle system to the corresponding PDE system is proven in arbitrary dimensions on the torus. The approximating particle scheme is based on a regularisation of the singular kernels. It is important to emphasise that our approach is fundamentally different in that it does not hinge on a regularisation argument. Instead it relies on identifying the particle system as a gradient flow and an in-depth treatment of particle-particle interactions. Not only are we able to circumnavigate the regularisation argument by taking into account particle collisions, but we also uncover and use the underlying gradient flow structure of the problem which, ultimately, provides existence and uniqueness. As a result we obtain a particle approximation of the system that is less restrictive in that it does not depend on the regularisation strength, albeit in one dimension and for a less singular interaction kernel. We expand upon this aspect in more detail in Remark \ref{rem:garroni}. 

\medskip
The paper is organised as follows.
\begin{itemize}
    \item In Section \ref{sec:preliminaries} we present the right setting for our problem, including our concept of weak measure solution for \eqref{eq:pde-sys-newtonian} in Definition \ref{def:solutions}, and provide some preliminary concepts related with one-dimensional optimal transport. 
    \item Section \ref{sec:grad_flow} is devoted to proving global-in-time existence and uniqueness of solutions to system \eqref{eq:ode-sys-newtonian},  using the theory of \textit{gradient flows} in Hilbert spaces. The main result of this section is the one in Lemma \ref{lem:subdiff_empty} proving that the sub-differential of the discrete functional is always non-empty in $\mathcal{C}^N\times \mathcal{C}^N$. The existence and uniqueness result in the discrete case is provided in Theorem \ref{thm:discrete_gradientflow}.
    \item Upon establishing well-posedness for the system of ODEs, we focus on some important properties of its solutions in Section \ref{sec:properties}. In Theorem \ref{thm:collisions_different_species} we provide explicit conditions describing the behavior of particles of opposite species after collision. In Theorem \ref{thm:collisions_same_species} we prove that particles of the same species can never collide. In Theorem \ref{thm:overlap} we provide an explicit solution to \eqref{eq:ode-sys-newtonian} in case the initial condition features a "cluster" or overlapping particles of the two species.
    \item Finally, in Section \ref{sec:many_particle_limit} we prove our many-particle approximation result. We show that the empirical measure of the particle system \eqref{eq:ode-sys-newtonian} converges in a suitable sense to the unique gradient flow solution to \eqref{eq:pde-sys-newtonian}. The main result is stated in Theorem \ref{thm:convergence_main_1}. The basic estimates needed for the proof are provided in Propositions \ref{prop:convergence_weak_Lm} and \ref{cor:dunford}.
\end{itemize}

\section{Preliminaries}\label{sec:preliminaries}
Throughout the paper we denote by $\mptr$ the set of  probability measures with finite second moment, i.e.,
$$
    \mptr=\left\{\mu\in\mP(\R) \, |\, m_2(\mu)<+\infty\right\},
    \mbox{ where } m_2(\mu)=\int_{\R}|x|^2\,d\mu(x).
$$
We use the symbol $\mptra$ to denote  the set of measures in $\mptr$ which are absolutely continuous with respect to the Lebesgue measure, i.e., $\mptra=\mP(\R)\cap L^1((1+|x|^2)\,dx)$. Next, for any measure $\mu\in\mP(\R)$ and a Borel map $T:\R\to\R$, we denote by $\nu = T_{\#}\mu$ the push-forward of $\mu$ through $T$, defined by
\begin{align*}
    \nu(A)&=\mu(T^{-1}(A)),\qquad \qquad \ \ \mbox{for any Borel set}\ A\subset\R,\\
    \mbox{or}\ \int_\R f(y)\,dT_{\#}\mu(y)&=\int_\R f(T(x))\,d\mu(x), \qquad \mbox{for any measurable function}\ f.
\end{align*}
Here, $T$ is usually referred to as \emph{transport map} pushing $\mu$ to $\nu$. Next, we equip the set $\mptr$ with the $2$-Wasserstein distance, which is defined for any $\mu,\nu\in \mptr$ as
\begin{equation}\label{wass}
W_2(\mu,\nu)=\left(\inf_{\gamma\in\Gamma(\mu,\nu)}\int_{\R^2}|x-y|^2\, d\gamma(x,y)\right)^{1/2},
\end{equation}
where $\Gamma(\mu,\nu)$ is the class of transport plans between $\mu$ and $\nu$, that is,
\begin{align*}
	\Gamma(\mu, \nu):= \{ \gamma\in \mP(\R^2)\,|\, \pi^1_{\#}\gamma = \mu, \,\pi^2_{\#}\gamma = \nu\},
\end{align*}
where $\pi^i:\R\times\R\rightarrow\R$, $i=1,2$, denotes the projection operator on the $i^\mathrm{th}$ component of the product space $\R^2$.  Setting $\Gamma_0(\mu,\nu)$ as the class of optimal plans, i.e., minimisers of \eqref{wass}, the (squared) Wasserstein distance can be written as
$$
	W_2^2(\mu,\nu)=\int_{\R^2}|x-y|^2\,d\gamma(x,y),
$$
for any $\gamma\in\Gamma_0(\mu,\nu)$. The set $\mptr$ equipped with the $2$-Wasserstein metric is a complete metric space which can be seen as a length space, see for instance \cite{AGS,S,V1,V2}. Since we are dealing with the evolution of two interacting species, we shall work on the product space $\mptr\times\mptr$ equipped with the 2-Wasserstein product distance defined via
$$
\mW_2^2(\gamma,\tilde \gamma)=W_2^2(\rho,\tilde \rho)+W_2^2(\eta,\tilde \eta),
$$
for all $\gamma=(\rho,\eta), \tilde \gamma=(\tilde \rho,\tilde \eta)$ in $\mptr\times\mptr$. Now, let us introduce a crucial tool for the one-dimensional case. For a given $\mu\in\mptr$ its cumulative distribution function is given by
\begin{equation}\label{eq:distribution-functions}
	F_\mu(x)=\mu((-\infty,x]).
\end{equation}
Since $F_\mu$ is a non-decreasing, right-continuous function such that
\begin{align*}
	\lim_{x\rightarrow -\infty} F_\mu(x) = 0, \quad \mbox{and} \quad 	\lim_{x\rightarrow +\infty} F_\mu(x) = 1,
\end{align*}
 we may define the pseudo-inverse function $X_\mu$ associated to $F_\mu$, by
\begin{equation}
	\label{eq:pseudoinverse_def}
	X_\mu(s):=\inf_{x\in \R}\{F_\mu(x)>s\},
\end{equation}
for any $s\in (0,1)$. It is easy to see that $X_\mu$ is right-continuous and non-decreasing as well. Having introduced the pseudo-inverse, let us now recall some of its important properties. First we notice that it is possible to pass from $X_\mu$ to $F_\mu$ as follows
\begin{equation}
	\label{eq:F_intermsof_X}
	F_\mu(x)=\int_0^1\mathds{1}_{(-\infty,x]}(X_\mu(s))\,ds=|\{X_\mu(s)\le x\}|.
\end{equation}
For any probability measure $\mu\in\mptr$ and the pseudo-inverse, $X_\mu$, associated to it, we have
\begin{equation}
	\label{eq:CoV_PseudoInverse}
	\int_\R f(x)\,d\mu(x)=\int_0^1f(X_\mu(s))\,ds,
\end{equation}
for every bounded continuous function $f$. Moreover, for $\mu,\nu\in\mptr$, the Hoeffding-Fr\'echet theorem \cite[Section 3.1]{RR} allows us to represent the 2-Wasserstein distance, $W_2(\mu,\nu)$, in terms of the associated pseudo-inverse functions via
\begin{equation}
	\label{eq:W2PseudoInverses}
	W_2^2(\mu,\nu)=\int_0^1|X_\mu(s)-X_\nu(s)|^2\,ds,
\end{equation}
since the optimal plan is given by $(X_\mu(s)\otimes X_\nu(s))_{\#}\mathcal{L}$, where $\mathcal{L}$ is the Lebesgue measure on the interval $[0,1]$, cf. also \cite{V1,CT}. We have seen that for every $\mu\in\mptr$ we can construct a non-decreasing $X_\mu$ according to \eqref{eq:pseudoinverse_def}, and by the change of variables formula \eqref{eq:CoV_PseudoInverse} we also know that $X_\mu$ is square integrable. Let us recall that this mapping is indeed a distance-preserving bijection between the space of probability measures with finite second moments and the convex cone of non-decreasing $L^2$-functions
\begin{equation}\label{eq:cone}
	\mathcal{C}:=\{f\in L^2(0,1)\,|\,f\ \mbox{is non-decreasing}\} \subset L^2(0,1).
\end{equation}
For $p\geq 1$, let us also introduce the \textit{$p$-Wasserstein} distance
\begin{equation}\label{eq:1-wass}
    W_p(\mu,\nu)=\inf_{\gamma\in\Gamma(\mu,\nu)} \left(\int_{\R^2}|x-y|^p\, d\gamma(x,y)\right)^{1/p},
\end{equation}
for any $\mu,\nu\in\mP_p(\R):=\{\mu\in\mP(\R)\ |\ m_p(\mu):=\int_\R|x|^p\,d\mu(x)<+\infty\}$. Since our problem is set in one space dimension, we have
\begin{equation}\label{eq:equiv-1-wass}
W_p(\mu,\nu)=\|X_\mu-X_\nu\|_{L^p([0,1])}.
\end{equation}
In the case $p=1$ we also have
\[W_1(\mu,\nu)=\|F_\mu - F_\nu\|_{L^1(\R)}.\]
We refer to \cite{V1,AGS,V2,S} for further details. The $p$-Wasserstein product distance is given by
$$
    \mW_p(\gamma,\tilde \gamma)=W_p(\rho,\tilde \rho)+W_p(\eta,\tilde \eta),
$$
for all $\gamma=(\rho,\eta), \tilde \gamma=(\tilde \rho,\tilde \eta)$ in $\mP_p(\R)\times\mP_p(\R)$.

For all $(\mu,\nu) \in \mptr\times\mptr$, we define the interaction energy functional
\[\mF(\mu,\nu)=-\frac{1}{2}\iint_{\R\times\R}|x-y|d\mu(y)d\mu(x)-\frac{1}{2}\iint_{\R\times\R}|x-y|d\nu(y)d\nu(x)+\iint_{\R\times\R}|x-y|d\mu(x)d\nu(y).\]
The following lemma is proven for absolutely continuous measures in \cite{CDFEFS}. Below, we extend it to general measures for the sake of self-containedness.

\begin{lem}\label{lem:functional_nonnegative}
For all $(\mu,\nu) \in \mptr\times\mptr$ we have
\[\mF(\mu,\nu)\geq 0.\]
\end{lem}

\begin{proof}
We first consider the case in which both $\mu$ and $\nu$ are absolutely continuous with respect to the Lebesgue measure and have continuous densities $\rho$ and $\eta$ respectively. In this case
\begin{align*}
     \mF(\mu,\nu)&= -\frac{1}{2}\iint_{\R\times\R}|x-y|\left(\rho(x)\rho(y)+\eta(x)\eta(y)-\rho(x)\eta(y)-\rho(y)\eta(x)\right)dx dy\\
    & =  -\frac{1}{2}\iint_{\R\times\R}|x-y|\sigma(x)\sigma(y) dx dy,\qquad\sigma=\rho-\eta.
\end{align*}
Recall that $N(x)=|x|$ satisfies $N'(x)=\sign(x)$ and $N''=2\delta_0$ in $\mathcal{D}'$. Therefore
\[\mF(\mu,\nu)=-\frac{1}{2}\int_{\R}N\ast\sigma (x)\delta_0\ast\sigma (x) dx =-\frac{1}{4}\int_{\R}N\ast\sigma (x)N''\ast\sigma (x) dx. \]
Integration by parts yields
\[\mF(\mu,\nu)=\frac{1}{4}\int_\R \left(N'\ast\sigma (x)\right)^2 dx -\frac{1}{4}\left[N\ast\sigma(x)N'\ast\sigma(x)\right]_{x=-\infty}^{x=+\infty}.\]
Now, arguing as in the proof of \cite[Lemma 3.7]{CDFEFS}, the boundary term at infinity vanishes due to the fact that $\rho$ and $\eta$ have finite second moment and $\sigma$ has zero average. Hence, $\mF(\mu,\nu)\geq 0$.

Let us now consider the general case $(\mu,\nu)\in \mptr\times\mptr$. Assume there exists a pair $(\bar\mu,\bar\nu)\in \mptr\times\mptr$ such that $\mF(\bar\mu,\bar\nu)<0$. By density of $(C(\R)\cap \mptr)^2$ in $\mptr^2$ with respect to the $2$-Wasserstein distance, there exist sequences $\rho_n, \eta_n\in C(\R)$ such that $W_2(\bar\mu,\rho_n)\rightarrow 0$ and $W_2(\bar\nu,\eta_n)\rightarrow 0$ as $n\rightarrow +\infty$. Consequently, $\rho_n\otimes\rho_n\rightarrow \bar\mu\otimes\bar\mu$ as $n\rightarrow +\infty$ in the weak measure sense, as well as $\eta_n\otimes\eta_n\rightarrow \bar\nu\otimes\bar\nu$ and $\rho_n\otimes\eta_n\rightarrow \bar\mu\otimes\bar\nu$ as $n\rightarrow +\infty$. Moreover, since $(x,y)\mapsto|x-y|$ has a sub-quadratic growth at infinity, by a standard cut-off argument, we have
\begin{align*}
    & \iint_{\R\times\R}|x-y|d\rho_n(x)d\rho_n(y)\rightarrow \iint_{\R\times\R}|x-y|d\bar\mu(x)d\bar\mu(y),\\
    & \iint_{\R\times\R}|x-y|d\eta_n(x)d\eta_n(y)\rightarrow \iint_{\R\times\R}|x-y|d\bar\nu(x)d\bar\nu(y),\\
    & \iint_{\R\times\R}|x-y|d\rho_n(x)d\eta_n(y)\rightarrow \iint_{\R\times\R}|x-y|d\bar\mu(x)d\bar\nu(y),
\end{align*}
which implies
\[
    \mF(\rho_n,\eta_n)\rightarrow \mF(\bar\mu,\bar\nu).
\]
On the other hand, the previous case implies
\[
    \mF(\rho_n,\eta_n)\geq 0, \qquad \hbox{for all $n$},
\]
which contradicts the assumption $\mF(\bar\mu,\bar\nu)<0$.
\end{proof}

We conclude this subsection by providing the rigorous concepts of solution to the continuum model \eqref{eq:pde-sys-newtonian} to be used in the many-particle limit. Such a concept of solution only refers to the case of absolutely continuous initial data.
\begin{defn}\label{def:solutions}
Let $(\rho_0,\eta_0)\in\mptra^2$. We say that the absolutely continuous curve $(\rho(\cdot),\eta(\cdot))\in C([0,+\infty)\,;\,\mptra^2)$ is a weak measure solution to \eqref{eq:pde-sys-newtonian} if, for all test functions $\varphi,\phi \in C^1_c([0,+\infty)\times \R)$ we have
\begin{subequations}\label{eq:weak_formulation}
    \begin{equation}
    \begin{split}
    - \int_0^{+\infty} &\int_\R \rho(x,t)\varphi_t(x,t)\, dx\, dt - \int_\R\rho_0(x)\varphi(x,0) dx\\
    & \ = \int_0^{+\infty}\iint_{\R\times\R} \rho(x,t) \rho(y,t)\sign(x-y)\varphi_x(x,t)\,dy\, dx\, dt\\
    & \quad \ - \int_0^{+\infty}\iint_{\R\times\R}\rho(x,t) \sign(x-y)\eta(y,t)\varphi_x(x,t)\,dy \,dx\, dt\,,
    \end{split}
    \end{equation}
and    
    \begin{equation}
    \begin{split}
     - \int_0^{+\infty} &\int_\R \eta(x,t)\phi_t(x,t)\, dx\, dt - \int_\R\eta_0(x)\phi(x,0)\, dx \\
    & \ = \int_0^{+\infty}\iint_{\R\times\R} \eta(x,t) \eta(y,t)\sign(x-y)\phi_x(x,t)\,dy\, dx\, dt\\
    & \quad \ - \int_0^{+\infty}\iint_{\R\times\R}\eta(x,t) \sign(x-y)\rho(y,t)\phi_x(x,t)\,dy\, dx dt.
    \end{split}
    \end{equation}
\end{subequations}

\end{defn}

\begin{rem}\label{rem:uniqueness}
The existence and uniqueness of solutions according to Definition \ref{def:solutions} follows from the existence and uniqueness result of gradient flow solutions proven in \cite{CDFEFS} and arguing as in \cite[Theorem 11.2.8]{AGS}.
\end{rem}

\section{Discrete Gradient Flow}\label{sec:grad_flow}
In this section we pose system \eqref{eq:ode-sys-newtonian} as gradient flow of a suitable discrete interaction energy functional. We shall denote by $x=(x_1,...,x_N)$ and $y=(y_1,...,y_N)$ the vectors corresponding to the particles of the two different species, where each particle $x_i$ has mass $m_i\in\R_+$ and $y_j$ has mass $n_j\in \R_+$, for all $i,j\in\{1,2,...,N\}$. Throughout, we shall work in the Hilbert space $(\RN\times\RN,\langle\cdot,\cdot\rangle_w)$, with the weighted scalar product defined by
$$
    \langle Z^1,Z^2\rangle_{w} := \sum_{i=1}^N m_ix_i^1 x_i^2+\sum_{j=1}^N n_j y_j^1 y_j^2,
$$
where $Z^1=(x^1,y^1),\ Z^2=(x^2,y^2)\in\RN\times\RN$. We drop the $w$-subscript in the definition of the weighted norm
\[\|Z\|^2=\langle Z,Z\rangle_w.\]

Since the problem is posed in one spatial dimension, we may label the particles such that they are monotonically ordered. Therefore, up to possibly relabelling, we may restrict the evolution to the convex cone $\mathcal{C}^N\times\mathcal{C}^N$, with
$$
   \mathcal{C}^N:=\left\{x\in\RN:x_1\le x_2\le...\le x_N\right\}.
$$

Since the analysis below requires a careful treatment of all particles of the two species, it is useful to introduce the following index notation that provides a precise way to label particles depending on the particles' relative location.
\begin{defn}[Index Notation]
    Given $Z=(x,y)\in \mathcal{C}^N\times \mathcal{C}^N$, for all $i \in\{1,\ldots,N\}$
    we define
    \begin{align*}
        & \sigma[x_i]=\left\{k\in \{1,\ldots,N\}\,:\,\,x_k=x_i\right\},\qquad \gamma[x_i]=\left\{j\in \{1,\ldots,N\}\,:\,\, y_j=x_i\right\},\\
        & \sigma^+[x_i]=\left\{k\in \{1,\ldots,N\}\,:\,\,x_k>x_i\right\},\qquad \gamma^+[x_i]=\left\{j\in \{1,\ldots,N\}\,:\,\, y_j>x_i\right\},\\
        & \sigma^-[x_i]=\left\{k\in \{1,\ldots,N\}\,:\,\,x_k<x_i\right\},\qquad \gamma^-[x_i]=\left\{j\in \{1,\ldots,N\}\,:\,\, y_j<x_i\right\},
    \end{align*}
    and for all $j\in \{1,\ldots, N\}$ we define
    \begin{align*}
        & \sigma[y_j]=\left\{h\in \{1,\ldots,N\}\,:\,\, y_h=y_j\right\}\,,\qquad
    \gamma[y_j]=\left\{i\in \{1,\ldots,N\}\,:\,\, x_i=y_j\right\},\\
        & \sigma^+[y_j]=\left\{h\in \{1,\ldots,N\}\,:\,\, y_h>y_j\right\}\,,\qquad
    \gamma^+[y_j]=\left\{i\in \{1,\ldots,N\}\,:\,\, x_i>y_j\right\},\\
     & \sigma^-[y_j]=\left\{h\in \{1,\ldots,N\}\,:\,\, y_h<y_j\right\}\,,\qquad
    \gamma^-[y_j]=\left\{i\in \{1,\ldots,N\}\,:\,\, x_i<y_j\right\}.
    \end{align*}
\end{defn}

\begin{rem}[Index Notation]
    Clearly, some of the sets $\sigma[x_i]$ and $\sigma[y_j]$ may be singletons. If a set $\sigma[x_i]$ contains more than one index, it means that the particle configuration $Z$ contains an $x$-cluster, i.e., a group of colliding particles of the $x$-species. Similarly, some of the sets $\gamma[x_i]$ and $\gamma[y_j]$ may be empty. A non-empty $\gamma[x_i]$ implies that there are particles of the $y$-species attached to $x_i$ in the $Z$ configuration. Moreover, note that the sets $\sigma^-[x_1]$, $\sigma^+[x_N]$, $\sigma^-[y_1]$, $\sigma^+[y_N]$ are always empty. 
\end{rem}

\begin{rem}[Discrete Fubini]
\label{rem:DiscreteFubini}
Throughout the main body we need to rearrange sums over indices of the particles involved in the dynamics. It is useful to highlight the following equality of index sets:
\begin{align*}
    \left\{ (i,j)\; \big| \; i < j \leq N,\; i = 1, \ldots, N \right\} = \left\{ (i,j)\; \big| \; 1 \leq i < j,\; j = 1, \ldots, N \right\},
\end{align*}
and therefore
\begin{align}
    \sum_{i=1}^N \sum_{j > i} Q_{i,j} = \sum_{j=1}^N \sum_{i < j} Q_{i,j} = \sum_{i=1}^N \sum_{j > i} Q_{j,i},
\end{align}
for any quantity $Q\in \R^{N\times N}$. Note that the first equality holds due to Fubini and the second one is due to swapping the roles of $i$ and $j$.
\end{rem}

\begin{lem}[Properties of Index Notation -- I]
\label{lem:PropertiesIndexNotationI}
For a given distribution of particles, $Z\in \mathcal{C}^N$ there exists $\epsilon_0>0$ such that for all $Z'\in \mathcal{C}^N$ with $|Z'-Z|_\infty< \epsilon_0/3$ there holds
\begin{align}
    \label{eq:inclusion_for_sigma}
    \sigma[x_i'] = \sigma[x_i'] \cap \sigma[x_i],
\end{align}
and the statement remains true when replacing $x_i$ by $y_i$. 
\end{lem}
\begin{proof}
The inclusion ``$\supset$'' is trivial and we only show the reverse inclusion ``$\subset$''. To this end, let $j \in \sigma[x_i']$ and assume that $j\notin \sigma[x_i]$. This implies that either $j\in \sigma^-[x_i]$ or $j\in \sigma^+[x_i]$. Assume, for instance, that $j\in \sigma^-[x_i]$, and therefore $x_j < x_i$ due to the fact that $Z$ is ordered. However, then
\begin{align}
    \epsilon_0 < x_i-x_j = x_i-x_i'+x_i' -x_j =x_i-x_i'+x_j' -x_j< 2\epsilon_0/3,
\end{align}
which is impossible. Similarly, $j\notin \sigma^+[x_i]$, which completes the proof.
\end{proof}

\begin{lem}[Properties of Index Notation -- II]
\label{lem:PropertiesIndexNotationII}
For a given distribution of particles, $Z \in \mathcal{C}^N$, there exists $\epsilon_0>0$ such that for all $Z' \in \mathcal{C}^N$ with $|Z'-Z|_\infty < \epsilon_0/3$ there holds
\begin{align}
    \label{eq:prop_sigma_minus}
    \begin{split}
    \sigma^-[x'_i] &= \sigma^-[x_i]\dot{\cup}\left(\sigma[x_i]\cap \sigma^-[x'_i]\right),\\
    \sigma^-[y'_i] &= \sigma^-[y_i]\dot{\cup}\left(\sigma[y_i]\cap \sigma^-[y'_i]\right),
    \end{split}
\end{align}
as well as
\begin{align}
    \label{eq:prop_sigma}
    \begin{split}
    \sigma[x'_i] &= \sigma[x_i] \setminus (\sigma^-[x'_i]\dot{\cup} \sigma^+[x'_i]),\\
    \sigma[y'_i] &= \sigma[y_i] \setminus (\sigma^-[y'_i]\dot{\cup} \sigma^+[y'_i]).
    \end{split}
\end{align}
Concerning the interspecies index sets, there holds
\begin{align}
    \label{eq:prop_gamma_plusminus}
    \begin{split}
    \gamma^\pm[x'_i] &=\gamma^\pm[x_i]\dot{\cup}(\gamma[x_i]\cap \gamma^\pm[x'_i])\\
    \gamma^\pm[y'_j] &=\gamma^\pm[y_j]\dot{\cup}(\gamma[y_j]\cap \gamma^\pm[y'_j]),
    \end{split}
\end{align}
as well as
\begin{align}
    \label{eq:prop_gamma}
    \begin{split}
    \gamma[x'_i] &=\gamma[x_i]\setminus (\gamma^-[x'_i]\dot{\cup}\gamma^+[x'_i])\\
    \gamma[y'_j] &= \gamma[y_j]\setminus (\gamma^-[y'_j]\dot{\cup}\gamma^+[y'_j]).
    \end{split}
\end{align}
\end{lem}
\begin{proof}
Let $Z \in \mathcal{C}^N$ be given. We set
\begin{align*}
    \epsilon_0 := \min \left\{|Z_k-Z_l| \; \big|\; 1 \leq k,l \leq 2N,\, \text{s.t. } Z_k \neq Z_l\right\} >0.
\end{align*}
Throughout, we assume that $Z'\in \mathcal{C}^N$, such that $|Z'-Z|_\infty < \epsilon_0/3$. We begin by proving statement \eqref{eq:prop_sigma_minus}.\\ ``$\subset$'':  Let $j \in \sigma^-[x_i']$. By definition this means that $x_j' < x_i'$. Since $Z' \in \mathcal{C}^N$ is ordered, we infer $j < i$. Since $Z$ is ordered, too, we have $x_j \leq x_i$, which means that
\begin{align*}
    j 
    &\in \sigma^-[x_i'] \cap \left(\sigma[x_i] \dot\cup \sigma^-[x_i]\right)\\
    &= \left(\sigma^-[x_i'] \cap \sigma[x_i]\right) \dot\cup  \left(\sigma^-[x_i'] \cap \sigma^-[x_i]\right)\\
    &\subset \left(\sigma^-[x'_i]\cap \sigma[x_i]\right) \dot{\cup} \sigma^-[x_i],
\end{align*}
where the last inclusion is due to fact that, without intersecting with $\sigma^-[x_i']$, the second set in the union, becomes larger, i.e., $\sigma^-[x'_i]\cap \sigma^-[x_i] \subset \sigma^-[x_i]$. \\
``$\supset$'': Conversely, let $j \in \sigma^-[x_i]\dot{\cup}\left(\sigma[x_i]\cap \sigma^-[x'_i]\right)$. Clearly, if $j$ belongs to the second set the statement is trivially satisfied. If, on the other hand, $j\in 
\sigma^-[x_i]$, we have $x_j < x_i$ and, due to the fact that both $Z$ and $Z'$ are ordered, we first obtain $j<i$ and therefore $x_j'\leq x_i'$. We conclude  $j \in \sigma^-[x_i'] \dot \cup \sigma[x_i']$. We will now show that $j\in \sigma[x_i']$ is impossible which implies the statement. The argument is by contradiction and we assume $j\in \sigma[x_i']$, i.e., $x_j' = x_i'$. However, by definition of $\epsilon_0>0$, this implies
\begin{align*}
    \epsilon_0 < x_i - x_j = x_i - x_i' + x_i' - x_j = x_i - x_i' + x_j' - x_j < 2\epsilon_0/3,
\end{align*}
where the closeness assumption $|Z'-Z|_\infty<\epsilon_0/3$ entered in the last inequality. Clearly this statement is absurd and therefore, $j\in \sigma^-[x_i']$ which completes the proof of the first statement. The same statement is true for the $y$-species using the same line of reasoning.\\
We continue with statement \eqref{eq:prop_sigma}.\\
``$\subset$'': If $j\in \sigma[x_i']$, then trivially, $j\notin \sigma^-[x_i'] \dot\cup \sigma^+[x_i']$ and it remains to show that $j \in \sigma[x_i]$. As before the argument is by contradiction and we assume, for instance, that $j \in \sigma^-[x_i]$. As before
\begin{align*}
    \epsilon_0 < x_i-x_j = x_i-x_i'+x_j'-x_j < 2\epsilon_0/3,
\end{align*}
ruling out the case  $j \in \sigma^-[x_i]$. Similarly, $j \notin \sigma^+[x_i]$ leaving as only possibility $j \in \sigma[x_i]$ which proves the inclusion.\\
``$\supset$'': Conversely, if $j\in\sigma[x_i] \setminus \left(\sigma^-[x_i'] \dot\cup \sigma^+[x_i']\right)$, there holds $x_j = x_i$ and $x_i'\leq x_j' \leq x_i'$. In particular, $x_i'=x_j'$, and therefore $j\in \sigma[x_i']$, which concludes the proof of this inclusion.\\
Next, we prove statement \eqref{eq:prop_gamma_plusminus}. We only focus on the ``-'' case, as the statement for ``+'' is given in a similar manner.  Let us begin with ``$\subset$'': Let $j \in \gamma^-[x_i']$, i.e., $y_j' < x_i'$. Assume, that $j \in \gamma^+[x_i]$, i.e., that $y_j > x_i$. Using these two inequalities yields
\begin{align*}
    \epsilon_0 < y_j - x_i = y_j - y_j' + y_j' - x_i \leq y_j - y_j' + x_i' - x_i < 2 \epsilon_0/3,
\end{align*}
implying that $j\in \gamma[x_i] \dot \cup \gamma^+[x_i]$, which yields the statement together with the fact that $j \in \gamma^-[x_i']$. \\
Regarding the opposite inclusion, ``$\supset$'', it suffices to show $\gamma^-[x_i] \subset \gamma^-[x_i']$ as the statement is trivially satisfied if $j$ is in the set $(\gamma[x_i]\cap \gamma^-[x'_i])$. Again, arguing by contradiction, let us assume $j \notin \gamma^-[x_i']$, i.e., $y_j'\geq x_i'$. In this case, we observe
\begin{align*}
    \epsilon_0 < y_j' - x_i' = y_j'-y_j + y_j - x_i \leq 2\epsilon_0/3,
\end{align*}
which yields the statement. 
Finally, we prove statement \eqref{eq:prop_gamma} beginning with ``$\subset$''.\\
Let $j \in \gamma[x_i']$, i.e., $y_j'=x_i'$, and assume that $j\notin \gamma[x_i]$, for instance, $y_j < x_i$. In this case
\begin{align*}
    \epsilon_0 < x_i-y_j = x_i-x_i' + y_j' - y_j \leq 2\epsilon_0/3,
\end{align*}
which is a contradiction. Similarly, we show that $y_j>x_i$ which shows the assertion. Finally, we show the reverse inclusion, ``$\supset$'':\\
In this case $x_i' \leq y_j' \leq x_i'$, i.e., $y_j' = x_i'$, and therefore $j \in \gamma[x_i']$. This concludes the proof of the lemma.
\end{proof}

Now, we introduce the discrete interaction energy functional acting on a given $Z=(x,y)\in\RN\times\RN$, as follows
\begin{equation}\label{eq:discrete-inter-en-funct}
    \mF[Z]=-\frac{1}{2}\sum_{i\neq j}m_i m_j|x_i-x_j|-\frac{1}{2}\sum_{i\neq j}n_i n_j|y_i-y_j|+\sum_{i,j}m_in_j|x_i-y_j| + \mathcal{I}_{\mathcal{C}^N}(x)+\mathcal{I}_{\mathcal{C}^N}(y),
\end{equation}
where $\mathcal{I}_{\mathcal{C}^N}$ is the indicator function of the cone $\mathcal{C}^N$, i.e.,
\begin{equation}
\begin{cases}
0 &\text{if}\ x\in \mathcal{C}^N,\\
+\infty &\text{otherwise}.
\end{cases}
\end{equation}
We shall often use the notation 
\begin{align*}
    & S(x):=-\frac{1}{2}\sum_{i\neq j}m_i m_j|x_i-x_j|, \qquad S(y):=-\frac{1}{2}\sum_{i\neq j}n_i n_j|y_i-y_j|,
\end{align*}
and
\begin{align*}
    & C(x,y):=\sum_{i,j}m_in_j|x_i-y_j|,
\end{align*}
so that
\[
\mF[Z] = S(x) + S(y) + C(x,y) + \mathcal{I}_{\mathcal{C}^N}(x)+\mathcal{I}_{\mathcal{C}^N}(y).
\]
$S$ represents the self-interaction part, i.e., interactions within the same species, while $C$ accounts for cross-interactions.

The functional $\mF$ is proper, i.e., $D(\mF)=\left\{Z\in\RN\times\RN:\mF[Z]<+\infty\right\}\neq\emptyset$, since we have
$$
    \mF[Z]\le|x_N|+|y_N|,
$$
for any $Z=(x,y)\in \mathcal{C}^N\times \mathcal{C}^N$. Moreover, note that the self-interaction part of the functional can also be rewritten as
$$
   S(x) + S(y)=-\sum_{i=1}^N\sum_{\{j:\,x_j>x_i\}}m_i m_j(x_j-x_i)-\sum_{i=1}^N\sum_{\{j:\,y_j>y_i\}}n_i n_j(y_j-y_i).
$$
\begin{rem}\label{rem:equivalent-expressions-self}
Note that the terms corresponding to the index $i=N$ give null contribution in the above sum. Nevertheless, we keep them in order to have a further expression for the self-interaction part of the functional that we will use later on in Eqs. \eqref{eq:functional_deeper_Self_x} and \eqref{eq:functional_deeper_Self_y}. Moreover, for the sake of completeness, let us point out the equivalent formulation
$$
   S(x) + S(y)=-\sum_{i=1}^N\sum_{\{j:\,x_j<x_i\}}m_i m_j(x_i-x_j)-\sum_{i=1}^N\sum_{\{j:\,y_j<y_i\}}n_i n_j(y_i-y_j),
$$
where the terms corresponding to the index $i=1$ give null contribution.
\end{rem}
The above observation allows us to prove the next lemma.
\begin{lem}\label{lem:convexity-functional}
The functional $\mathcal{F}:\RN\times\RN\to\R$ is convex.
\end{lem}
\begin{proof}
Take $Z^1=(x^1,y^1),\ Z^2=(x^2,y^2)\in\RN\times\RN$ and a convex combination between them $Z^\alpha=\alpha Z^1+(1-\alpha)Z^2=(\alpha x^1+(1-\alpha)x^2, \alpha y^1+(1-\alpha)y^2)=(x^\alpha,y^\alpha)$, with $\alpha\in[0,1]$. We need to prove
$$
\mF[\alpha X^1+(1-\alpha)X^2]\le\alpha\mF[X^1]+(1-\alpha)\mF[X^2].
$$
If either $Z^1\not\in \mathcal{C}^N\times \mathcal{C}^N$ or $Z^2\not\in \mathcal{C}^N\times \mathcal{C}^N$ then the above inequality is trivial since the right-hand side is $+\infty$. When both $Z^1$ and $Z^2$ are in $\mathcal{C}^N\times \mathcal{C}^N$, then so is $Z^\alpha$ as this set is convex. Hence, the convexity of $\mF$ can be checked, as follows, by means of the order-preserving property in $\mathcal{C}^N\times \mathcal{C}^N$, 
\begin{align*}
    \mF[Z^\alpha]&=-\sum_{i=1}^N\sum_{j\in \sigma^+[x_i^\alpha]} m_i m_j(x_j^\alpha-x_i^\alpha) - \sum_{i=1}^N\sum_{j\in \sigma^+[y_i^\alpha]}n_i n_j(y_j^\alpha-y_i^\alpha)+\sum_{i,j}m_in_j|x_i^\alpha-y_j^\alpha|\\
    &=-\alpha\sum_{i=1}^N\sum_{j\in\sigma^+[x_i^1]}m_i m_j(x_j^1-x_i^1)-(1-\alpha)\sum_{i=1}^N\sum_{j\in \sigma^+[x_i^2]}m_i m_j(x_j^2-x_i^2) \\
    &\quad -\alpha\sum_{i=1}^N\sum_{j\in\sigma^+[y_i^1]}n_i n_j(y_j^1-y_i^1)-(1-\alpha)\sum_{i=1}^N\sum_{j\in\sigma^+[y_i^2]} n_i n_j(y_j^2-y_i^2)\\
    &\quad +\sum_{i,j} m_in_j|\alpha (x_i^1-y_j^1)+(1-\alpha)(x_i^2-y_j^2)|,
\end{align*}

and the assertion follows by using the triangle inequality in the last term.
\end{proof}

We shall now investigate in greater detail the functional $\mF$. Our next goal is to provide an expression of $\mF[Z]$ for $Z\in \mathcal{C}^N\times \mathcal{C}^N$ accounting for a possible superposition of groups of particles.

We now rewrite the functional $\mF[Z]$ using the above index notation and Remark \ref{rem:equivalent-expressions-self}. Let us start with the self-interaction part:
\begin{align}
    \label{eq:functional_deeper_Self_x}
    \begin{split}
    S(x) &=-\sum_{i=1}^N\sum_{j:x_i>x_j}m_i m_j (x_i - x_j) \\
    & = \sum_{i=1}^N m_i x_i \left[ -\sum_{j\in \sigma^-[x_i]}m_j + \sum_{j\in \sigma^+[x_i]}m_j\right].
    \end{split}
\end{align}
A similar expression may be obtained for the $y$-part:
\begin{equation}\label{eq:functional_deeper_Self_y}
    S(y)=\sum_{j=1}^N n_j y_j \left[- \sum_{i\in \sigma^-[y_j]}n_i + \sum_{i\in \sigma^+[y_j]}n_i\right].
\end{equation}
We now consider the cross-interaction term
\begin{align}
    & C(x,y)=\sum\sum_{x_i>y_j}m_i n_j(x_i-y_j) + \sum\sum_{x_i<y_j}m_i n_j(y_j-x_i)\nonumber\\
    & \ = \sum_{i=1}^N m_i x_i\left[\sum_{j\in \gamma^-[x_i]} n_j -\sum_{j\in \gamma^+[x_i]}n_j\right]
    + \sum_{j=1}^N n_j y_j \left[  \sum_{i\in \gamma^-[y_j]} m_i-\sum_{i\in \gamma^+[y_j]} m_i\right]. \label{eq:functional_deeper_Cross}
\end{align}
In order to deal with gradient flows in Hilbert spaces, we need to introduce the concept of \textit{Fr\'echet sub-differential}. We adapt the definition of this classical concept to our specific case.
\begin{defn}[Fr\'echet sub-differential]\label{def:frechet}
For a given proper, convex, and lower semi-continuous functional  $\mF$ on $\RN\times\RN$, we say that $P\in\RN\times\RN$ belongs to the sub-differential of $\mF$ at $Z\in \R^{N}\times\R^{N}$ if and only if
\begin{equation}\label{eq:subdiff1}
\mF[Z']-\mF[Z]\ge\langle P,Z'-Z\rangle_w,
\end{equation}
for all $Z'\in \R^N\times\R^N$. The sub-differential of $\mF$ at $Z$ is denoted by $\partial \mF(Z)$, and if $\partial \mF(Z)\neq \emptyset$ then we denote by $\partial^0 \mF(Z)$ the element of minimal (weighted) norm of $\partial \mF(Z)$.
\end{defn}

\begin{rem}\label{rem:subdiff_littleo}
We recall that, since $\mF$ is convex, requiring condition \eqref{eq:subdiff1} to be satisfied for all $Z'\in \R^N\times\R^N$ can be relaxed to 
\begin{equation}\label{eq:subdiff2}
    \mF[Z']-\mF[Z]\ge\langle P,Z'-Z\rangle_w + o(\|Z'-Z\|), \qquad \hbox{as $Z'\rightarrow Z$}.
\end{equation}
\end{rem}

As $\mF[Z]$ attains the  value $+\infty$ outside the cone $\mathcal{C}^N\times \mathcal{C}^N$, it is reasonable to assume $Z\in\mathcal{C}^N\times \mathcal{C}^N$ as a necessary condition to have  $\partial\mF(Z)\neq\emptyset$. In the one species case (see \cite[Proposition 2.10]{BCDFP}) one can actually prove that being in the cone is a necessary \emph{and sufficient} condition to have a non-empty sub-differential. Such a property is non-trivial in the many species case. We provide it in the next lemma.

\begin{lem}\label{lem:subdiff_empty}
Let $Z=(x,y)\in\RN\times\RN$. Then $\partial\mF(Z)\neq\emptyset$ if and only if $Z\in \mathcal{C}^N\times \mathcal{C}^N$. 
\end{lem}
\begin{proof}
Similarly to \cite[Proposition 2.10]{BCDFP}, let us assume $Z=(x,y)\not\in \mathcal{C}^N\times \mathcal{C}^N$. Without restriction, we assume for instance $x\not\in \mathcal{C}^N$, which implies $I_{\mathcal{C}^N}(x)=+\infty$. Assuming $P\in\partial\mF[Z]$, we have
$$
    \mF[Z']-S(x)-S(y)-C(x,y)\ge\langle P,Z'-Z\rangle_w+I_{\mathcal{C}^N}(x),
$$
for all $Z'=(x',y')\in\R^{2N}$. In particular, the previous inequality would hold for any  $Z'\in \mathcal{C}^N\times \mathcal{C}^N$, which is clearly a contradiction since, in this case, the left-hand side is finite, while the right-hand side is infinite. 

Let us now assume that $Z\in \mathcal{C}^N\times \mathcal{C}^N$. The inequality \eqref{eq:subdiff1} is trivially satisfied for an arbitrary $P$ in case $Z'\not\in \mathcal{C}^N\times \mathcal{C}^N$, therefore we can assume without restriction that $Z'=(x',y')\in \mathcal{C}^N\times \mathcal{C}^N$. Our goal is to show that there exists a vector $P\in \R^N\times \R^N$ such that \eqref{eq:subdiff2} holds as $Z'\rightarrow Z$. Therefore, without restriction we assume that $\|Z'-Z\|<\varepsilon_0$ for some $\varepsilon_0>0$ to be chosen later on.

We now compute
\begin{align*}
    & S(x')-S(x) \\
    & \ = \sum_{i=1}^N m_i \left[x'_i\left(-\sum_{j\in \sigma^-[x'_i]}m_j + \sum_{j\in \sigma^+[x'_i]}m_j  \right) - x_i \left(-\sum_{j\in \sigma^-[x_i]}m_j + \sum_{j\in \sigma^+[x_i]}m_j \right)\right]\\
    & \ = \sum_{i=1}^N m_i (x'_i-x_i)\left(-\sum_{j<i}m_j+\sum_{j>i} m_j  \right) + R,
\end{align*}
with
\begin{align}
     R 
    &= \sum_{i=1}^N m_i x'_i \left( \sum_{j<i\,:\,\,j\in \sigma[x'_i]}m_j - \sum_{j> i\,:\,\,j\in \sigma[x_i']}m_j\right)\\
    & \qquad + \sum_{i=1}^N m_i x_i \left(-\sum_{j<i\,:\,\, j\in \sigma[x_i]}m_j + \sum_{j> i\,:\,\, j\in \sigma[x_i]}m_j\right)\\
    &=: R_1 - R_2.
\end{align}
With 
\begin{align}
    R_1 = \sum_{i=1}^N  \sum\limits_{\substack{j<i \\ j \in \sigma[x_i']}} m_j m_i x_i' - \sum_{i=1}^N \sum\limits_{\substack{j<i \\ j \in \sigma[x_i]}}m_j  m_i x_i, 
\end{align}
and
\begin{align}
    R_2 = \sum_{i=1}^N  \sum\limits_{\substack{j>i \\ j \in \sigma[x_i']}} m_j m_i x_i' - \sum_{i=1}^N \sum\limits_{\substack{j>i \\ j \in \sigma[x_i]}}m_j  m_i x_i.
\end{align}
Using the fact that 
$$
    \sigma[x_i] = (\sigma[x_i] \cap \sigma[x_i']) \,\dot \cup\, (\sigma[x_i] \setminus\sigma[x_i']),
$$
we may split the second sum and simplify the term $R_1$, i.e.,
\begin{align}
    R_1 &= \sum_{i=1}^N  \sum\limits_{\substack{j<i \\ j \in \sigma[x_i']}}m_j m_i x_i' - \sum_{i=1}^N \sum\limits_{\substack{j<i \\ j \in \sigma[x_i]\cap \sigma[x_i']}}m_j  m_i x_i - \sum_{i=1}^N  \sum\limits_{\substack{j<i \\ j \in \sigma[x_i]\setminus \sigma[x_i']}}m_j m_i x_i \\
    &= \sum_{i=1}^N m_i (x_i'-x_i) \sum\limits_{\substack{j<i \\ j \in \sigma[x_i']}}m_j - \sum_{i=1}^N m_ix_i \sum\limits_{\substack{j<i \\ j \in \sigma[x_i]\setminus \sigma[x_i']}}m_j, 
\end{align}
having used Eq. \eqref{eq:inclusion_for_sigma} of Lemma \ref{lem:PropertiesIndexNotationI} in the last line. In the same vein, we have
\begin{align}
    R_2 &= \sum_{i=1}^N \sum\limits_{\substack{j>i \\ j \in \sigma[x_i']}}m_j  m_i x_i' - \sum_{i=1}^N  \sum\limits_{\substack{j>i \\ j \in \sigma[x_i]\cap \sigma[x_i']}} m_j m_i x_i - \sum_{i=1}^N  \sum\limits_{\substack{j>i \\ j \in \sigma[x_i]\setminus \sigma[x_i']}}m_j m_i x_i\\
    &= \sum_{i=1}^N m_i (x_i'-x_i) \sum\limits_{\substack{j>i \\ j \in \sigma[x_i']}}m_j - \sum_{i=1}^N m_i x_i \sum\limits_{\substack{j>i \\ j \in \sigma[x_i]\setminus \sigma[x_i']}}m_j.
\end{align}
Upon subtraction, we obtain
\begin{align}
    \label{eq:difference_Rone_Rtwo}
    R_1-R_2 &= \sum_{i=1}^N m_i (x_i'-x_i) \left( \sum\limits_{\substack{j<i \\ j \in \sigma[x_i']}}m_j -\sum\limits_{\substack{j>i \\ j \in \sigma[x_i']}}m_j\right)+ R_3,
\end{align}
where
\begin{align}
    \label{eq:Rthree}
    R_3 = \sum_{i=1}^N \sum\limits_{\substack{j>i \\ j \in \sigma[x_i]\setminus \sigma[x_i']}}m_j  m_i x_i - \sum_{i=1}^N \sum\limits_{\substack{j<i \\ j \in \sigma[x_i]\setminus \sigma[x_i']}}m_j m_i x_i.
\end{align}
Rearranging the sum according to Remark \ref{rem:DiscreteFubini}, the first term can be rewritten,
\begin{align}
    \sum_{i=1}^N \sum\limits_{\substack{j>i}} \chi_{ \sigma[x_i]\setminus \sigma[x_i']}(j) m_j  m_i x_i = \sum_{j=1}^N \sum_{i<j} \chi_{ \sigma[x_i]\setminus \sigma[x_i']}(j) m_j  m_i x_i.
\end{align}
Since 
$$
    j \in \sigma[x_i]\setminus \sigma[x_i'] \Longleftrightarrow i \in \sigma[x_j]\setminus \sigma[x_j'],
$$
we can simplify the expression further by relabelling, i.e., switching the roles of $i$ and $j$, to obtain
\begin{align}
    \sum_{i=1}^N \sum\limits_{\substack{j<i}} \chi_{ \sigma[x_i]\setminus \sigma[x_i']}(j) m_j  m_i x_i &= \sum_{i=1}^N \sum_{j>i} \chi_{ \sigma[x_i]\setminus \sigma[x_i']}(j) m_i  m_j x_j \\
    &= \sum_{i=1}^N \sum\limits_{\substack{j<i\\ j \in \sigma[x_i]\setminus \sigma[x_i']}} m_i  m_j x_j.
\end{align}
Substituting the simplified expression into the first term of $R_3$, i.e., Eq. \eqref{eq:Rthree}, we obtain
\begin{align}
    R_3 = \sum_{i=1}^N \sum\limits_{\substack{j<i\\ j\in \sigma[x_i]\sigma[x_i']}} m_i m_j (x_j-x_i) = 0.
\end{align}
Thus, revisiting Eq. \eqref{eq:difference_Rone_Rtwo}, we get
\begin{align}
    R_1-R_2 &= \sum_{i=1}^N m_i (x_i'-x_i) \left( \sum\limits_{\substack{j<i \\ j \in \sigma[x_i']}}m_j -\sum\limits_{\substack{j>i \\ j \in \sigma[x_i']}}m_j\right).
\end{align}
The right-hand side is shown to vanish changing the labels $i,j$ and using Remark \ref{rem:DiscreteFubini}.
We have therefore proven
\begin{subequations}
\label{eq:subdiff_self}
\begin{equation}
        S(x')-S(x) = \sum_{i=1}^N m_i (x'_i-x_i)\left(\sum_{j>i}m_j-\sum_{j<i}m_j\right),
\end{equation}
as well as
\begin{equation}
        S(y')-S(y) = \sum_{j=1}^N n_j (y'_j-y_j)\left(\sum_{i>j}n_i-\sum_{i<j}n_j\right).
    \end{equation}
\end{subequations}
Due to \eqref{eq:functional_deeper_Cross}, we have
\begin{align*}
    & C(x',y')-C(x,y)\\
    & \ =\sum_{i=1}^N m_i (x'_i-x_i)\left(\sum_{j\in \gamma^-[x_i]}n_j -\sum_{j\in \gamma^+[x_i]}n_j\right) + \sum_{j=1}^N n_j (y'_j-y_j)\left(\sum_{i\in \gamma^-[y_j]}m_i -\sum_{i\in \gamma^+[y_j]}m_i\right) + \widetilde{R},
\end{align*}

with
\begin{align}
    \widetilde R = \widetilde R_1 + \widetilde R_2 + \widetilde R_3 + \widetilde R_4,
\end{align}
where 
\begin{align}
    \widetilde R_1 = \sum_{i=1}^N m_i x_i' \left(\sum_{j \in \gamma^-[x_i']} n_j - \sum_{j\in \gamma^-[x_i]}n_j \right), \quad \mbox{and} \quad \widetilde R_2 = \sum_{i=1}^N m_i x_i' \left(\sum_{j \in \gamma^+[x_i]} n_j - \sum_{j\in \gamma^+[x_i']}n_j \right)
\end{align}
as well as
\begin{align}
    \widetilde R_3 = \sum_{j=1}^N n_j y_j' \left( \sum_{i \in \gamma^-[y_j']}m_i - \sum_{i\in \gamma^-[y_j]}m_i\right), \quad\mbox{and}\quad \widetilde R_4 = \sum_{j = 1}^N n_j y_j'\left(\sum_{i\in \gamma^+[y_j]} m_i - \sum_{i \in \gamma^+[y_j']}m_i\right).
\end{align}
Using Eq. \eqref{eq:prop_gamma_plusminus} of Lemma \ref{lem:PropertiesIndexNotationII}, we may write
\begin{align}
    \widetilde R_1&= \sum_{i=1}^N m_i x_i' \sum_{j \in \gamma[x_i] \cap \gamma^-[x_i']} n_j  \\
    &\geq \sum_{i=1}^N \sum_{j =1}^N \chi_{\gamma[x_i] \cap \gamma^-[x_i']}(j) m_i n_j y_j'\\
    &= \sum_{j=1}^N \sum_{i =1}^N \chi_{\gamma[y_j] \cap \gamma^+[y_j']}(j) m_i n_j y_j'\\
    &= \sum_{j=1}^N \sum_{i \in \gamma[y_j] \cap \gamma^+[y_j']}^N m_i n_j y_j',
\end{align}
where the inequality is due to the fact that $j \in \gamma^-[x_i']$ and thus $x_i' > y_j'$ and the penultimate line is by rearranging terms in the sum since
\begin{align}
    j \in \gamma[x_i] \cap \gamma^-[x_i] \Leftrightarrow \left(x_i = y_j \mbox{ and } x_i' > y_j'\right) \Leftrightarrow i \in \gamma[y_j] \cap \gamma^+[y_j'].
\end{align}
Using a similar argument, we see that 
\begin{align}
    \widetilde R_3 = \sum_{j=1}^N  \sum_{i\in \gamma[y_j]\cap \gamma^-[y_j']} m_i n_j y_j' \geq  \sum_{i=1}^N \sum_{j\in \gamma[y_i]\cap \gamma^+[x_i']} m_i  n_j x_i'.
\end{align}
Finally, we note that, upon using Eq. \eqref{eq:prop_gamma_plusminus} of Lemma \ref{lem:PropertiesIndexNotationII}, we may write
\begin{align}
    \widetilde R_2 =- \sum_{i=1}^N \sum_{j\in \gamma[x_i] \cap \gamma^+[x_i']} m_i n_j x_i', \quad\mbox{and}\quad \widetilde R_4 = -\sum_{j=1}^N \sum_{i\in \gamma[y_j]\cap \gamma^+[y_j]} m_i  n_j y_j'.
\end{align}
Combining the terms, we have $\widetilde R\geq  0$.
The above estimate, together with \eqref{eq:subdiff_self}, implies that the vector $P=(p,q)\in \R^N\times \R^N$ with
\begin{align*}
    & p=(p_i)_{i=1}^N\,,\qquad q=(q_j)_{j=1}^N,\\
    & p_i = -\sum_{j<i}m_j + \sum_{j>i}m_j + \sum_{j\in \gamma^-[x_i]}n_j - \sum_{j\in \gamma^+[x_i]}n_j,\\
    & q_j =-\sum_{i<j}n_i + \sum_{i>j}n_i + \sum_{i\in \gamma^-[y_j]}m_i - \sum_{i\in \gamma^+[y_j]}m_i,
\end{align*}
satisfies \eqref{eq:subdiff2}, i.e., $P\in \partial\mathcal{F}(Z)$.
\end{proof}

The functional $\mF$ defined in \eqref{eq:discrete-inter-en-funct} is proper, continuous, and convex on the Hilbert space $(\R^N\times\R^N,\langle\cdot,\cdot\rangle_w)$, in view of Lemma \ref{lem:convexity-functional}. As a consequence of the previous properties we have $\partial\mF$ is a maximal monotone operator. Hence, we can use the theory of Br\'{e}zis \cite[Theorem 3.1]{Brezis}, e.g. in the form stated in \cite[Section 9.6, Theorem 3]{evans} in order to pose system \eqref{eq:ode-sys-newtonian} as the gradient flow associated to \eqref{eq:discrete-inter-en-funct}. 

\begin{defn}
\label{def:discrete_gradflow}
Let $Z_0=(x_0,y_0)\in \mathcal{C}^N\times\mathcal{C}^N$. An absolutely continuous curve $(x(t),y(t))\in \R^{2N}$ is a \emph{gradient flow} for the functional $\mF$ if $Z(t):=(x(t),y(t))$ is a Lipschitz function on $[0,+\infty)$, i.e., $\frac{dZ}{dt}\in L^\infty([0,+\infty);\RN\times\RN)$ (in the sense of distributions) and if it satisfies the sub-differential inclusion
	\begin{equation}\label{eq:subdiff_inclusion}
	-\begin{pmatrix}
\dot x(t)\\
\dot y(t)
	\end{pmatrix}\in\partial\mF(Z(t)),
	\end{equation}
for almost every $t\in [0,+\infty)$ with $(x(0),y(0))=(x_0(\cdot),y_0(\cdot))$.
\end{defn}
Resorting to the theory of Br\'ezis, we get the following theorem.
\begin{thm}\label{thm:discrete_gradientflow}
Let $Z_0=(x_0,y_0)\in \mathcal{C}^N\times \mathcal{C}^N$ be an initial datum. Then, there exists a unique solution $Z(t)=(x(t),y(t))$ in the sense of Definition \ref{def:discrete_gradflow} such that $Z(0)=Z_0$. Moreover, given $Z^1_0, Z^2_0 \in \mathcal{C}^N\times \mathcal{C}^N$ and the two corresponding solutions $Z^1(t), Z^2(t)$ in the sense of Definition \ref{def:discrete_gradflow} with initial data $Z^1_0$ and $Z^2_0$ respectively, the stability property
\begin{equation*}
    \|Z^1(t)-Z^2(t)\| \leq \|Z^1_0-Z^2_0\|,
\end{equation*}
holds for all $t\geq 0$.
\end{thm}

\begin{rem}
Lemma \ref{lem:subdiff_empty} affects the statement of the above Theorem \ref{thm:discrete_gradientflow} in that the class of initial conditions for which existence and uniqueness of a gradient flow solution holds in the sense of Definition \ref{def:discrete_gradflow} coincides with \emph{the whole convex cone} $\mathcal{C}^N\times \mathcal{C}^N$. According to \cite[Theorem 3.1]{Brezis}, initial data should belong to the \emph{domain of the sub-differential} of $\mathcal{F}$ in order to have existence and uniqueness of solutions. Lemma \ref{lem:subdiff_empty} assures that $\mathcal{C}^N\times \mathcal{C}^N$ and the domain of $\partial\mathcal{F}$ are the same set.
Moreover, the result in Lemma \ref{lem:subdiff_empty} also provides an explicit expression $P(p,q)$ of at least one element in the sub-differential at any given configuration of particles that includes possible collisions, both within the same species and among particles of opposite species. Such expression anticipates what we shall see in the next section regarding the behavior of particles in presence of superpositions/collisions. For example, the $i$-th particle of the first species of a given particle configuration is subject to two \emph{self-repulsive} drifts, the former due to the accumulated mass of particles with label $<i$ (pointing to the positive direction), the latter combining the amount of mass possessed by particles with label $>i$ (pointing to the negative direction), \emph{regardless of possible superpositions}. At the same time, the $i$-th particle is subject to cross-attractive drifts depending on the particles of the opposite species. Particles of the $y$-species located strictly to the left of $x_i$ contribute to a cross-attractive drift pointing to the negative direction, whereas particles of the $y$ species posed strictly to the right of $x_i$ cause $x_i$ move in the positive direction. Particles of the $y$-species whose location coincides with $x_i$ do not contribute to the cross-interaction part of $P$.
\end{rem}

\begin{rem}
Note that the minimal selection in the sub-differential in \eqref{eq:subdiff_inclusion} is achieved as consequence of the convexity of $\mF$, cf. \cite{evans,CDFEFS}, for instance. More precisely, the quoted result in \cite{Brezis} implies $Z$ admits a right derivative for every $t\in[0,+\infty)$ and
$$
    -\frac{d^+Z}{dt}(t)=\partial^0\mF(Z(t)),
$$
for every $t\in[0,+\infty)$. Here,
\[
    \partial^0\mF(Z)=\argmin\left\{\|P\|\,:\,\, P\in \partial\mF(Z)\right\}.
\]
Moreover the function $t\mapsto\partial^0\mF(X(t))$ is right-continuous and the function $t\mapsto\left\|\partial^0\mF(Z(t))\right\|$ is non-increasing.
\end{rem}

\section{Qualitative Properties of the ODEs System}\label{sec:properties}
Having established a well-posedness theory for system \eqref{eq:ode-sys-newtonian}, let us now focus on some important properties of the solutions. In particular we are interested in the dynamic of collisions and in the support of the solution. 

The main results of this section are obtained under the assumption that all particles have the same mass, i.e., $1/N$. Similar results may be obtained for more general masses. We highlight this in Remark \ref{rem:different_masses}. Since the main goal of this work is the convergence of the particle approximation scheme, we shall henceforth focus on the case of equal masses.

\subsection{Collisions between particles of different species}

In this subsection we discuss collisions between any two particles $x_i$ and $y_j$ of opposite species, for $i,j\in\{1,2,...,N\}$. Let us denote by
\begin{equation}
    \label{eq:collision-time}
    t_*=\inf\{t\geq 0:x_i(t)=y_j(t)\ \mbox{for some}\ 1 \leq i,j\leq N\}.
\end{equation}
In Subsection \ref{subsec:collisions} we shall see that, indeed, $t_*<+\infty$. Since all trajectories are continuous and the number of particles is finite, there exists a time $t^*$ such that the above $\inf$ is achieved. Let $i_0,j_0 \in \{1,2,...,N\}$ such that $x_{i_0}(t_*) = y_{j_0}(t_*)$.  The following theorem covers all possible configurations of the colliding particles right after the collision time $t_*$. 

\begin{thm}\label{thm:collisions_different_species}

Let $t_*$ be a collision time defined in \eqref{eq:collision-time} and $i_0,j_0$ be such that  $x_{i_0}(t_*) = y_{j_0}(t_*)$. Assume that all other particles occupy a different position at $t=t_*$. Then there exists $\epsilon>0$ such that $\theta(t):=x_{i_0}(t) - y_{j_0}(t)$ satisfies 
\begin{enumerate}
    \setlength\itemsep{0.5em}
    \item $\theta(t) = x_{i_0}(t) - y_{j_0}(t) > 0, \quad\hbox{if}\quad i_0 > j_0$,
    \item $\theta(t) = x_{i_0}(t) - y_{j_0}(t) < 0, \quad\hbox{if}\quad i_0 < j_0$,
    \item $\theta(t) = x_{i_0}(t) - y_{j_0}(t) = 0, \quad\hbox{if}\quad i_0 = j_0$,
\end{enumerate}
for all $t \in (t_*, t_* + \epsilon)$. Moreover, in case $i_0=j_0$,
\begin{itemize}
\item the two particles $x_{i_0}$ and $y_{i_0}$ remain attached for all $t\geq t^*$,
\item the two particles $x_{i_0}$ and $y_{i_0}$ have zero velocity on $[t^*,t^*+\epsilon]$.
\end{itemize}
\end{thm}

\begin{proof}
At time $t_*$ we consider a particle closest to $x_{i_0}(t_*) = y_{j_0}(t_*)$, denoted by $z_k$, where $z_k = x_i$, for $i\neq i_0$ or $z_k=y_j$, for $j\neq j_0$. Let us denote $d_* := |z_k(t_*) - x_{i_0}(t_*)|$. Since no particle is moving at a speed larger than $2$, cf. \eqref{eq:ode-sys-newtonian}, there exists an $0 < \epsilon < d_*/4$, such that their distance remains strictly positive, for all $t\in [t_*, t_* + \epsilon]$. In particular this means that $x_{i_0}, y_{j_0}$ remain the only particles in the interval $[x_{i_0}(t_*) - d_*/2, x_{i_0}(t_*) + d_*/2]$ for any $t \in [t_*, t_*+ \epsilon]$.
As a consequence, the sign of $\theta(t)$ is either strictly positive or strictly negative on $(t_*, t_*+ \epsilon)$, or it is equal to zero, since the velocities remain constant.

If $\theta(t)\neq 0$, no superpositions occur for the whole time interval $(t_*,t_*+\epsilon)$. As 
the functional $\mF$ is $C^1$ on a configuration without superpositions, the sub-differential of $\mF$ is single-valued and corresponds to the right-hand side of \eqref{eq:ode-sys-newtonian}. Therefore, for all $t_*^+\in (t_*,t_*+\epsilon)$ we have

\begin{equation}
    \begin{cases}
        \dot x_{i_0}(t_*^+)= \displaystyle \sum_{k\in \sigma^-[x_{i_0}]}\frac{1}{N} - \sum_{k\in \sigma^+[x_{i_0}]} \frac{1}{N} - \sum_{k\in \gamma^-[x_{i_0}]} \frac{1}{N} + \sum_{k\in \gamma^+[x_{i_0}]}\frac{1}{N},\\[0.7em]
        \dot y_{j_0}(t_*^+)= \displaystyle \sum_{k\in \sigma^-[y_{j_0}]}\frac{1}{N} - \sum_{k\in \sigma^+[y_{j_0}]}\frac{1}{N} - \sum_{k\in \gamma^-[y_{j_0}]}\frac{1}{N} + \sum_{k\in \gamma^+[y_{j_0}]}\frac{1}{N}.
    \end{cases}
\end{equation}
We now compare the velocities of the two particles $x_{i_0}$ and $y_{j_0}$ in the two configurations (1) and (2) with their velocity. In case $\theta>0$ on $(t*,t^*+\epsilon)$ we have, for some $t_*^+\in (t_*,t_*+\epsilon)$,
\begin{equation}\label{eq:first_case}
    \dot x_{i_0}(t_*^+)>\dot y_{j_0}(t_*^+),
\end{equation}
which is equivalent to 
\begin{align*}
    &\sum_{\{k:x_k<x_{i_0}\}}\frac{1}{N}-\sum_{\{k:x_k>x_{i_0}\}}\frac{1}{N}-\sum_{\{k:y_k<x_{i_0}\}}\frac{1}{N}+\sum_{\{k:y_k>x_{i_0}\}}\frac{1}{N}\\
    &\quad > \sum_{\{k:y_k<y_{j_0}\}}\frac{1}{N}-\sum_{\{k:y_k>y_{j_0}\}}\frac{1}{N}-\sum_{\{k:x_k<y_{j_0}\}}\frac{1}{N}+\sum_{\{k:x_k>y_{j_0}\}}\frac{1}{N}.
\end{align*}
After multiplying by $N$, the above condition reads
\begin{align*}
    (i_0-1) \,-\, (N-i_0) \,-\, j_0 \,+\, (N-j_0) &> (j_0-1) \,-\,  (N-j_0) \,-\, (i_0-1) \,+\, (N -i_0+1),
\end{align*}
which, upon simplification, is equivalent to
\begin{align*}
    i_0-j_0&>\frac{1}{2},
\end{align*}
which, in turn, is equivalent to $i_0>j_0$, due to the fact that $i_0,j_0\in\{1,...,N\}$.  A similar computation yields that in case $\theta(t)<0$ on the interval $(t_*,t_*+\epsilon)$ then
 $\dot x_{i_0}(t_*^+)<\dot y_{j_0}(t_*^+)$ on some time $t_*^+$, and then
$$
    i_0-1-(N-i_0)-(j_0-1)+N-j_0+1<j_0-1-(N-j_0)-i_0+(N-i_0)\iff i_0<j_0.
$$
Clearly, in case $i_0=j_0$ none of the two above situations are possible, and we must necessarily have that the two particles $x_{i_0}$ and $y_{j_0}$ overlap on the time interval $[t^*,t^*+\epsilon)$. In order to determine the speed of the two particles in this case, we observe that, in the particle configuration in which $x_{i_0}=y_{i_0}$ and all other particles occupy different positions, the $i_{0}$-th component of sub-differential $P=(p,q)$ found in Lemma \ref{lem:subdiff_empty} reads
\[p_{i_0}=q_{i_0}=(i_0-1)-(N-i_0) -(i_0-1)+(N-i_0) = 0,\]
and by uniqueness of the gradient flow solution according to Definition \ref{def:discrete_gradflow} the two particles have zero velocity on the time interval $[t^*,t^*+\epsilon]$ in which they do not collide with other particles. 
\end{proof}

\begin{rem}\label{rem:different_masses}
In case of different masses, the inequality in \eqref{eq:first_case} reads
\begin{align}
    M_{i_0-1}-(1-M_{i_0})-N_{j_0}+1-N_{j_0} > N_{j_0-1}-(1-N_{j_0})-M_{i_0-1}+(1-M_{i_0-1}),
\end{align}
which gives the more general condition  $m_{i_0}>4N_{j_0-1}+3n_{j_0}-4M_{i_0-1}$ for $x_{i_0}(t)>y_{j_0}(t)$. Here $M_{i_0}=m_1+...+m_{i_0}$ and  $N_{j_0}=n_1+...+n_{j_0}$. 
\end{rem}

\subsection{Collisions between particles of the same species} Theorem \ref{thm:collisions_different_species} covers all possible types of  collisions between two particles of opposing species. This subsection is dedicated to investigating whether or not two particles of the same species can collide.

\begin{thm}
\label{thm:collisions_same_species}
Assume the particles $x_1,\ldots,x_N$ do not overlap initially and that $m_1 = \dots = m_N = n_1 = \dots = n_N = 1/N$. Then, particles of the $x$-species never overlap for all times $t> 0$. The same statement holds for particles of the $y$-species.
\end{thm}
\begin{proof}
Arguing by contradiction, let us assume there exists a time $t_*$ such that $x_i(t_*)=x_{i+1}(t_*)$. Without loss of generality we may choose such $t_*$ as the first collision time for those two particles. Still without losing generality, we assume there exists an $\epsilon>0$ such that no other collisions involving either $x_i$ or $x_{i+1}$ occur on $(t_*-\epsilon,t_*)$. Clearly, there exists $t_*^- \in (t_* - \epsilon, t_*)$ such that
\begin{equation}
    \label{eq:contr_inequality}
    \dot x_{i+1}(t_*^-)<\dot x_i(t_*^-).
\end{equation}
We shall cover all the possible cases.

\textbf{Case 1: $x_i$ and $x_{i+1}$ collide ``without any particles of the opposite species strictly between them''.} This case also covers the situation in which one or more particles of the $y$-species collide with $x_i$ and $x_{i+1}$ at $t_*$ but none of them are set strictly between $x_i$ and $x_{i+1}$ on the above time interval. Hence, there exists an index $j$ such that both $x_i$ and $x_{i+1}$ have exactly $j$ $y$-particles on their left and $N-j$ on their right on the same time interval $(t_*-\epsilon,t_*)$. Hence, denoting by $M_i=m_1+...+m_i$ and $N_j=n_1+...+n_j$ for any $i,j\in\{1,2,...,N\}$, inequality \eqref{eq:contr_inequality} implies
$$
    M_i-(1-M_{i+1})-N_j+1-N_j<M_{i-1}-(1-M_i)-N_j+1-N_j\iff m_i+m_{i+1}<0,
$$
which is clearly false.

\textbf{Case 2: as in Case 1 but with more colliding particles such that none of them "strictly between $x_i$ and $x_{i+1}$"}. This situation can be covered as in Case 1. All particles moving strictly outside the interval $[x_i,x_{i+1}]$ prior to the collision do not affect the computations in Case 1.

\textbf{Case 3: one $y$-particle is set "strictly between $x_i$ and $x_{i+1}$" before collision.} Assume now there is an index $j$ such that $x_i(t_*)=x_{i+1}(t_*)=y_j(t_*)$ and $x_i(t)<y_j(t)< x_{i+1}(t)$ for all $t\in (t_*-\epsilon,t_*)$. In this case, there must be a time $t_*^-$ such that \eqref{eq:contr_inequality} is satisfied since $y_j$ slows down $x_{i+1}$ and attracts $x_i$. The explicit computation of the velocities yields in this case
$$
    M_i-(1-M_{i+1})-N_j+1-N_j<M_{i-1}-(1-M_i)-N_{j-1}+1-N_{j-1}\iff m_i+m_{i+1}<2n_j.
$$
Since we are assuming that all particles have the same mass $1/N$ the above is a contradiction.

\textbf{Case 4: one $y$-particle is attached to either $x_i$ or $x_{i+1}$ before the collision time $t_*$.} Assume now that we are in the same situation as in Case 3 except that the particle $y_j$ is attached to $x_{i+1}$ on the time interval $(t_*-\epsilon,t_*)$. From Theorem \ref{thm:collisions_different_species} we know that this is possible only if $j={i+1}$ and $\dot{x}_{i+1}=\dot{y}_j = 0$ on $(t_* - \epsilon, t_*)$. On the other hand, with the notation of Case 1, $\dot{x}_i$ is explicitly computed on $t\in (t_*-\epsilon,t_*)$,
\[\dot{x}_i(t)=M_{i-1} - (1-M_i) -N_i+1-N_i,\]
and since all particles have mass $1/N$ we deduce
\[\dot{x}_i(t)=-1/N,\]
which clearly shows that $x_i$ and $x_{i+1}$ cannot collide at time $t_*$ in this case. We remark that this also covers the situation in which one or more particles of the $y$ species are set strictly between $x_i$ and $x_{i+1}=y_j$ before time $t_*$.

\textbf{Case 5: more than one $y$-particle is set between $x_i$ and $x_{i+1}$.} Assume now that $x_i$ and $x_{i+1}$ collide having two or more particles of the $y$ species, say $y_{j},\ldots,y_{j+k}$ with $k\geq 1$, strictly between them in the time interval $(t_*,t_*+\epsilon)$. In this case, at least two particles of the $y$-species collide without particles of the $x$-species strictly between them, and this is impossible due to the first case we considered for the $x$-species (with reversed roles).
\end{proof}

We can collect the information in Theorems \ref{thm:collisions_same_species} and \ref{thm:collisions_different_species} as follows.

\begin{cor}\label{cor:collisions}
Assume the particles $x_1,\ldots,x_N$ do not overlap initially, and assume the same holds for the particles $y_1,\ldots,y_N$. Then, particles of the same species never collide for all times. Particles of opposite species can only meet in a binary collision. When that occurs, they behave according to the three cases stated in Theorem \ref{thm:collisions_different_species}.
\end{cor}

\begin{rem}\label{rem:bouncing}
The results in Theorem \ref{thm:collisions_different_species} and Corollary \ref{cor:collisions} clearly show that there can be no \emph{bouncing} in the particle system \eqref{eq:ode-sys-newtonian}, i.e., a particle cannot  reach another particle and then remain strictly before it after touching it. This is immediate in the case of particles of the same species as they simply never collide. As for the case of particles of opposite species, assume $x_i$ reaches $y_j$ at time $t^*$. If they touch each other and are then bounced back, this would imply that their post-collisional velocities are the same as their pre-collisional ones, hence, for instance, $\dot{x}_i(t)>\dot{y}_j(t)$ for $t>t^*$, but this is in contradiction with $x_i(t)<y_j(t)$ for $t>t^*$, recalling that $x_i(t^*)=y_j(t^*)$.
\end{rem}

\subsection{Initial overlapping.}

The results in the previous two subsections are relevant in case of no initial overlapping of particles. In this subsection we analyse the situation of an initial ``cluster'' involving particles of both species. 
We prove the following result.

\begin{lem}\label{lem:initial_cluster}
Assume there exist non-negative integers $0\leq h,k,n,m\leq N$ with $h<k<n<m$ and some $\lambda\in \R$ such that
\begin{align*}
    & x_{i}(0)=\lambda, \qquad \hbox{for all $i=h+1,\ldots,n$},\\
    & y_{j}(0)=\lambda, \qquad \hbox{for all $j=k+1,\ldots,m$},
\end{align*}
and assume no other particles of the $x$ or $y$ species occupy the position $\lambda$ at time $t=0$. Assume further that no particles other than $x_i$ with $i=h+1,\ldots,n$ and $y_j$ with $j=k+1,\ldots,m$ overlap at $t=0$. Then, for $t>0$ and prior to the next collision, the following holds:
\begin{align}
   & x_{h+1}(t)<\ldots<x_{k}(t)<\lambda,\nonumber\\
   & x_{k+1}(t)\equiv\ldots\equiv x_{n}(t)\equiv \lambda,\nonumber\\
   & y_{k+1}(t)\equiv\ldots\equiv y_{n}(t)\equiv \lambda,\nonumber\\
   & \lambda<y_{n+1}(t)<\ldots<y_{m}(t).\label{eq:particles_overlapping1}
\end{align}
\end{lem}

\begin{proof}
We prove the assertion by providing an explicit particle trajectory
\[
    Z(t)=(x(t),y(t))=(x_1(t),\ldots,x_N(t),y_1(t),\ldots,y_N(t)),
\]
satisfying \eqref{eq:particles_overlapping1} for all $t\in [0,\epsilon]$ for a suitably small $\epsilon$. To perform this task, we will prove that our chosen particle trajectory satisfies the sub-differential inclusion \eqref{eq:subdiff_inclusion} for all $t\in (0,\epsilon)$ for a suitably small $\epsilon>0$. The result then follows by uniqueness, see Theorem \ref{thm:discrete_gradientflow}. For simplicity we adopt the notation
\[x_i(0)=\bar{x}_i\,,\qquad y_j(0)=\bar{y}_j,\]
for all $i,j=1,\ldots,N$. Moreover, we use the notation
\[I:=\{1,\ldots,N\}\,,\qquad J:=\{k+1,\ldots,n\}.\]
By assumption, particles $x_i$ and $y_j$ with $i=1,\ldots,h,n+1,\ldots,N$ and $j=1,\ldots,k,m+1,\ldots,N$ occupy distinct positions at $t=0$, therefore we  verify that they move as follows, for $t\in [0,\epsilon)$ and $\epsilon>0$ small enough such that no collisions arise in $(0,\epsilon)$:
\begin{align*}
     & x_i(t) =\bar{x}_i+\frac{1}{N}[i-1 - (N-i)-\#\{j\in I\,:\,\, \bar{y}_j<\bar{x}_i\} + \#\{j\in I\,:\,\,\bar{y}_j>\bar{x}_i\}]t,\\
    & y_j(t)=\bar{y}_j + \frac{1}{N}[j-1-(N-j) - \#\{i\in I\,:\,\, \bar{x}_i<\bar{y}_j\} + \#\{i\in I\,:\,\,\bar{x}_i>\bar{y}_j\}]t.
\end{align*}
Moreover, we set
\begin{align*}
    & x_i(t)=\lambda -\frac{1}{N}(2(k-i)+1)t\,,\qquad \hbox{for $i=h+1,\ldots,k$},\\
    & y_j(t)=\lambda +\frac{1}{N}(2(j-n)-1)t\,,\qquad \hbox{for $j=n+1,\ldots,m$},\\
    & x_i(t)\equiv \lambda \,,\qquad \hbox{for $i=k+1,\ldots,n$},\\
    & y_j(t)\equiv \lambda \,,\qquad \hbox{for $j=k+1,\ldots,n$},
\end{align*}
for $t\in [0,\epsilon)$ for $\epsilon>0$ small enough such that $x_{h+1}$ does not collide with $x_h$ and $y_{m}$ does not collide with $y_{m+1}$, which is guaranteed by the fact that $|\bar{y}_{m+1}-\bar{y}_m|>0$ and $|\bar{x}_{h+1}-\bar{x}_h|>0$. For a fixed $t\in (0,\epsilon)$ we prove that the vector
\begin{align*}
    & -\dot{Z}(t)=-(\dot{x}(t),\dot{y}(t))=-(p,q)\\
    & p=(p_i)_{i=1}^N\,,\qquad q= (q_j)_{j=1}^N\\
    & p_i=\dot{x}_i(t)\,,\qquad q_j=\dot{y}_j(t).
\end{align*}
belongs to $\partial\mF(Z(t))$. For simplicity we denote $x=x(t)$ and $y=y(t)$. By means of a direct computation we obtain
\begin{align*}
  & \mF(x',y')-\mF(x,y)\\
  & \ = \frac{1}{N^2}\sum_{i\in I\setminus J} (x_i'-x_i)[-(i-1)+N-i]+\frac{1}{N^2}\sum_{j\in I\setminus J}(y_j'-y_j)[-(j-1)+N-j]\\
  & \quad +\frac{1}{N^2}\sum_{i\in J}(x_i'-x_i)[-k+(N-n)] +\frac{1}{N^2}\sum_{i\in J}x_i'[-\#\{h\in J\,:\,\,x_h'<x_i'\}+\#\{h\in J\,:\,\,x_h'>x_i'\}]\\
  & \quad -\frac{1}{N^2}\sum_{i\in J}x_i\underbrace{[-\#\{h\in J\,:\,\,x_h<x_i\}+\#\{h\in J\,:\,\,x_h>x_i\}]}_{=0}\\
  & \quad +\frac{1}{N^2}\sum_{j\in J}(y_j'-y_j)[-k+(N-n)] +\frac{1}{N^2}\sum_{j\in J}y_j'[-\#\{k\in J\,:\,\,y_k'<y_j'\}+\#\{k\in J\,:\,\,y_k'>y_j'\}]\\
  & \quad -\frac{1}{N^2}\sum_{j\in J}y_j\underbrace{[-\#\{k\in J\,:\,\,y_k<y_j\}+\#\{k\in J\,:\,\,y_k>y_j\}]}_{=0}\\
  & \quad + \frac{1}{N^2}\sum_{i\in I\setminus J}(x_i'-x_i)[\#\{k\in I\,:\,\,y_k<x_i\}-\#\{k\in I\,:\,\,y_k>x_i\}]\\
  & \quad + \frac{1}{N^2}\sum_{i\in J}(x_i'-x_i)[k-(N-n)] +\frac{1}{N^2}\sum_{i\in J}x_i'[\#\{k\in J\,:\,\,y_k'<x_i'\}-\#\{k\in J\,:\,\,y_k'>x_i'\}]\\
  & \quad -\frac{1}{N^2}\sum_{i\in J}x_i\underbrace{[\#\{k\in J\,:\,\,y_k<x_i\}-\#\{k\in J\,:\,\,y_k>x_i\}]}_{=0}\\
  & \quad +\frac{1}{N^2}\sum_{j\in I\setminus J}(y_j'-y_j)[\#\{h\in I\,:\,\,x_h<y_j\}-\#\{h\in I\,:\,\,x_h>y_j\}]\\
  & \quad +\frac{1}{N^2}\sum_{j\in J}(y_j'-y_j)[k-(N-n)]+\frac{1}{N^2}\sum_{j\in J}y_j'[\#\{h\in J\,:\,\,x_h'<y_j'\}-\#\{h\in J\,:\,\,x_h'>y_j'\}]\\
  & \quad -\frac{1}{N^2}\sum_{j\in J}y_j\underbrace{[\#\{h\in J\,:\,\,x_h<y_j\}-\#\{h\in J\,:\,\,x_h>y_j\}]}_{=0}.
\end{align*}
Combining all the terms, and recalling that for small $t$
\begin{align*}
    & \#\{k \in I\,:\,\, y_k<x_i\}-\#\{k\in I\,:\,\,y_k>x_i\} = 2k-N\qquad \hbox{for $i\in \{h+1,\ldots,k\}$}\\
    & \#\{h \in I\,:\,\, x_h<y_j\}-\#\{h\in I\,:\,\,x_h>y_j\} = 2n-N\qquad \hbox{for $j\in \{n+1,\ldots,m\}$}
\end{align*}
we obtain
\begin{align*}
    & \mF(x',y')-\mF(x,y) \\
    & \ =-\frac{1}{N^2}\sum_{i\in\{1,\ldots,h,n+1,\ldots,N\}}(x_i'-x_i)[i-1-(N-i)-\#\{j\in I\,:\,\,y_i<x_i\}+\#\{j\in I\,:\,\,y_j>x_i\}]\\
    & \ - \frac{1}{N^2}\sum_{j\in\{1,\ldots,k,m+1,\ldots,N\}}(y_j'-y_j)[j-1-(N-j)-\#\{i\in I\,:\,\,x_i<y_j\}+\#\{i\in I\,:\,\, x_i>y_j\}]\\
    & \ + \frac{1}{N^2}\sum_{i=h+1}^k(x_i'-x_i)[2(k-i)+1]+ \frac{1}{N^2}\sum_{j=n+1}^m(y_j'-y_j)[2(n-j)+1]
    \end{align*}
    \begin{align*}
    & \ - \frac{1}{N^2}\sum\sum_{i,h\in J\,:\,\,x_h'<x_i'} x_i' +\frac{1}{N^2} \sum\sum_{i,h\in J\,:\,\,x_h'>x_i'} x_i'-\frac{1}{N^2}\sum\sum_{j,k\in J\,:\,\,y_k'<y_j'}y_j' +\frac{1}{N^2} \sum\sum_{j,k\in J\,:\,\,y_k'> y_j'}y_j'\\
    & \ + \sum\sum_{i,k\in J\,:\,\,y_k'<x_i'} x_i' - \sum\sum_{i,k\in J\,:\,\,y_k'>x_i'} x_i' + \sum\sum_{j,h\in J\,:\,\,x_h'<y_j'} y_j' - \sum\sum_{j,h\in J\,:\,\,x_h'>y_j'} y_j'.
\end{align*}
We now observe that the last eight terms above can be rewritten as
\[-\frac{1}{2N^2}\sum_{i\in J}\sum_{h\in J} |x_i'-x_h'|-\frac{1}{2N^2}\sum_{j\in J}\sum_{k\in J} |y_j'-y_k'| + \frac{1}{N^2}\sum_{i\in J}\sum_{j\in J} |x'_i-y_j'|,\]
which equals the functional $\mF$ computed on the particle configuration  $(x_i)_{i\in J},\ (y_j)_{j\in J}$. Due to Lemma \ref{lem:functional_nonnegative}, these terms amount to a non-negative quantity. This proves the assertion.
\end{proof}

\begin{thm}\label{thm:overlap}
Assume there exist non-negative integers $i,j,h,k$ such that
\[x_i(0)=x_{i+1}(0)=\ldots=x_{h-1}(0)=x_{h}(0)= y_j(0)=y_{j+1}(0)=\ldots=y_{k-1}(0)=y_k(0),\]
and no other particles occupy the same position. Then, 
\begin{itemize}
\item all particles (of both species) having indeces in the set $\{i,\ldots,h\}\cap\{j,\ldots,k\}$ remain in the same position for all $t\geq 0$.
\end{itemize}
Moreover, on some time  interval $[0,\epsilon]$ for a suitably small $\epsilon>0$,
\begin{itemize}
\item if $i<j$, particles $x_i,\ldots,x_{j-1}$ detach from $x_j$ moving to the left, and are strictly ordered;
\item if $i>j$, particles $y_j,\ldots,y_{i-1}$ detach from $y_i$ moving to the left, and are strictly ordered;
\item if $h>k$, particles $x_{k+1},\ldots,x_h$ detach from $x_k$ moving to the right, and are strictly ordered;
\item if $k>h$, particles $y_{h+1},\ldots,y_k$ detach from $y_h$ moving to the right, and are strictly ordered.
\end{itemize}
\end{thm}

\begin{proof}
Assume first there is only one group of particles occupying the same position initially as in the hypothesis of the theorem. The case in which $k\geq h$ and $i\leq j$ is covered in Lemma \ref{lem:initial_cluster}. The symmetric case $k\leq h$ and $i\geq j$ is analogous. The case in which either 
\begin{itemize}
\item $i<j$ and $h>k$
\end{itemize}
or
\begin{itemize}
\item $i>j$ and $h<k$
\end{itemize}
can be proven similarly. We omit the details.

The general case in which there is more than one cluster of overlapping particles follows by similar computations as in Lemma \ref{lem:initial_cluster} in case there are ``excess'' particles of one species on one side and ``excess'' particles of the opposite species on the other side, and by similar considerations as the one above in this proof for the general case. The computations are quite long and tedious but do not include any additional technical difficulty to the one in Lemma \ref{lem:initial_cluster} and are therefore left to the reader.
\end{proof}

As a consequence of the result in Theorem \ref{thm:overlap}, we are able to describe in full detail the short-time solution of the particle system \eqref{eq:ode-sys-newtonian} in case of initial overlapping of particles:
\begin{itemize}
    \item [(1)] If particles of the same species, for instance $x_i,\ldots,x_{i+h}$ occupy the same position initially and no particles of the other species are in the same initial position, then the particles ``scatter apart'': they immediately detach and move apart, with $x_i(t)<\ldots<x_{i+h}(t)$. The same situation occurs in the one species case, see \cite{BCDFP}.
    \item [(2)] If particles of the two species occupy  the same position initially, their behaviour depends on the cumulative mass of each species on that position. More precisely, particles of the two species featuring the \emph{same index} are stationary for all times and remain attached to the initial cluster. The remaining particles move away from the cluster and ``disperse'' towards the particles of the opposite species with the same indexes.
    \item [(3)] As a special case of case (2), let us highlight that, if no particles have the same index, no stationary cluster is formed and all particles diffuse.
    \item [(4)] In case (2), the whole particle system is split into two independent particle sub-systems separated by the initial cluster. Each of the two sub-system is only subject to the interaction energy $\mF$ restricted to it. The two sub-systems are totally decoupled. 
\end{itemize}

\subsection{Support of the solution}
The solution to system \eqref{eq:ode-sys-newtonian} is a pair $(x(t),y(t))\in \mathcal{C}^N\times \mathcal{C}^N$, i.e., $2N$ particles of two opposing species distributed on the real line such that
\begin{align*}
    x_1(t)\le x_2(t)\le...\le x_N(t), \qquad \mbox{and} \qquad y_1(t)\le y_2(t)\le...\le y_N(t),
\end{align*}
for $t\ge0$. Hence, we may consider the support of the solution as the time-dependent interval $[a(t),b(t)]$, where
\begin{align*}
    a(t)=\min\{x_1(t), y_1(t)\}, \qquad \mbox{and} \qquad b(t)=\max\{x_N(t), y_N(t)\},
\end{align*}
for $t\ge0$. An interesting property concerning the support is that it is determined by the initial datum in the following way.

\begin{prop}\label{prop:support_solution}
Let $[a(t),b(t)]$ be the support of the solution $(x(t),y(t))$ to system \eqref{eq:ode-sys-newtonian} in the sense of Definition \ref{def:discrete_gradflow} with an initial datum $(x_0,y_0)\in \mathcal{C}^N\times \mathcal{C}^N$. Then $[a(t),b(t)]\subseteq[a(0),b(0)]$.
\end{prop}
\begin{proof}
Let us assume without loss of generality $a(0)=x_1(0)$ and $b(0)=y_N(0)$. Assume first that $x_1$ is the only particle occupying the position $a(0)$ at $t=0$. In this case it is easily seen that $\dot{x}_1(0)=-1+1/N+1\geq 0$, by Eq. \eqref{eq:ode-sys-newtonian}, and the particle moves to the right. Assume now that a cluster of particles $x_1,\ldots,x_h$ and $y_1,\ldots,y_k$ occupy the position $a(0)$ at time $t=0$. The result in Theorem \ref{thm:overlap} implies that particles of both species with indices in $\{1,\ldots,\min\{h,k\}\}$ remain at $a(0)$ for all times, whereas the remaining particles move towards the positive direction. In particular $\dot{x}_1=0$.  

Hence, in any case, $\dot{x}_1(t)\geq 0$ until $x_1$ collides with another particle. Similarly, one can prove that $\dot{y}_N\leq 0$ until $y_N$ collides with another particle. When the first collision occurs, we have a cluster of particles and we can re-apply Theorem \ref{thm:overlap} and conclude that $x_1$ will stop moving for all times. An analogous statement holds for $y_N$. The assertion is therefore proven.
\end{proof}

The uniform bound for the support of the particle system proven above has an important repercussion on the sub-differential of the functional $\mF$:

\begin{prop}\label{prop:uniform_subdiff}
Let $T\geq 0$ be fixed and let $Z(\cdot)=(x(\cdot),y(\cdot))$ be the unique gradient flow solution to \eqref{eq:ode-sys-newtonian} according to Definition \ref{def:discrete_gradflow}. Then, there exists a constant $C\geq 0$ independent of $N$ and only depending on the diameter of the initial support of the particles such that 
\[\sup\left\{\|P\|_w\,:\,\, P=(p,q)\in \partial\mF(Z(t))\,,\,\, t\in [0,T]\right\}\leq C.\]
\end{prop}

\begin{proof}
From Definition \ref{def:frechet} with $Z'=Z(t)+P$ we obtain
\[\langle P,P\rangle_w \leq \mF(Z(t)+P)-\mF(Z(t)).\]
Now, a simple triangle inequality implies
\[\mF(Z(t)+P) \leq \frac{1}{N^2}\sum_{i,j}|x_i(t)+p_i-y_j(t)-q_j| \leq  \frac{1}{N^2}\sum_{i,j}|x_i(t)-y_j(t)| + \frac{1}{N^2}\sum_{i,j}(|p_i| + |q_j|)\]
and due to Proposition \ref{prop:support_solution} and Young inequality we get
\[\mF(Z(t)+P) \leq C_1 + \frac{\|P\|_w^2}{2}\]
for some $C_1\geq 0$ only depending on the diameter of the initial support of the particles. By a similar estimate we get $-\mF(Z(t))\leq C_2$, where $C_2\geq 0$ once again only depends on the diameter of the initial support of the particles. Therefore, we get
\[\|P\|_w^2 \leq C_1+C_2 + \frac{\|P\|_w^2}{2},\]
and the assertion follows.
\end{proof}

\subsection{Number of collisions and long time behaviour}\label{subsec:collisions}

The results in the previous Theorem \ref{thm:overlap} emphasise that some initial conditions may imply the formation of ``mixed-clusters'', i.e., groups of particles of both species that are stationary in time and split the whole set of particles into groups that move independently. As these groups can be considered  as separate gradient flows of the same energy functional (up to rescaling the mass), in order to understand the long-time behaviour of our system we can assume without loss of generality that no 
``mixed-clusters'' are formed.

The case in which no such mixed clusters form immediately after time $t=0$ occurs in one of the following three cases:
\begin{itemize}
    \item [(1)] There is no superposition of particles initially;
    \item [(2)] The only superposition consists of particles of the same species;
    \item [(3)] There is an initial superposition of particles of mixed species but no particles of two opposite species and same index occupy the same position.
\end{itemize}
Therefore, without loss of generality we shall assume we are in one the above three situations. 

The first thing we emphasise here is that the velocity of each particle is the same between two consecutive collisions, and we know that collisions cannot occur among particles of the same species, see Theorem \ref{thm:collisions_same_species}. Hence, only collisions with the opposite species can occur, according to Theorem \ref{thm:collisions_different_species}. Clearly, velocities  change after a collision only in case of \emph{crossings}, and Remark \ref{rem:bouncing} shows that  crossings are only  possible after a collision. Now, the result in Theorem \ref{thm:collisions_different_species} shows that two particles $x_i$ and $y_j$ collide and cross if and only if $i>j$. After the collision, $x_i$  slows down by $2/N$, as it has one more particle of the $y$-species attracting it form the left, and one less particle of the $y$-species attracting it from the right. Nothing will change with respect to the interaction with particles of the $x$-species. 

A  simple computation allows to compare velocities of two particles $x_h$ and $y_k$ of opposite species even when they are far apart, showing that $\dot{x}_h(t)-\dot{y}_k(t)$ is always positive when $h>k$, and of course $x_h(t)<y_k(t)$. Indeed, we have the following bounds for the respective velocities:
\begin{align*}
    &\dot x_h(t)\geq h-1-N+h-(k-1)+N-k+1=2h-2k+1,\\
    &\dot y_k(t)\leq k-1 - N + k - h+ N-h = 2k -2h-1.
\end{align*}
This easily implies
\[
\dot x_h(t)-\dot y_k(t)\geq 4h - 4k +2\ge0 \iff h\geq k,
\]
since $h,k\in\mathbb{N}$. Hence, $x_i$ may continue colliding with the next particle of the $y$-species, namely $y_{j+1}$, having crossed $y_j$, provided $i$ is also larger than $j+1$, and so on, for a certain number of times, $n$, until $i=j+n$. Then the two particles will collide and stick together for all times. 

As a consequence of that, we have an explicit control on the total number $N$ of collisions involving a given particle. More precisely, every particle with index $i$ can collide with at most  $N_i=i$ particles, this number being reached, for example, if  $x_i$ has $y_1,\ldots,y_i$ on its right. Hence, the maximum possible number of collisions is
\[
    N_{\mathrm{max}} = 2\sum_{i=1}^N i = N (N+1).
\]
Since the minimum relative velocity between two consecutive particles that will collide is $2/N$, the largest possible time between two consecutive collisions is of order
\[\Delta_{\mathrm{max}}\sim \frac{N}{2}. \]

Hence, one expects that particles will reach the ``stationary solution''
\[
    x_i=y_i\,,\qquad \hbox{for all $i=1,\ldots,N$},
\]
by a time of order $N^3$ at the latest.

\section{The "continuum" gradient flow as a many-particle limit}\label{sec:many_particle_limit}
In this section we deal with the rigorous derivation of system \eqref{eq:pde-sys-newtonian} as a many-particle limit of a system of the form \eqref{eq:ode-sys-newtonian}. 

We will pursue this task for general probability measures as initial conditions. In order to single out the mathematical difficulties arising from the case of singular measures as initial conditions, we will start for simplicity with the case of an absolutely continuous initial condition $\rho_0, \eta_0 \in \mptra$ with compact support. We recall that the support of a measure $\mu\in\mptr$ is the closed set
$$
\text{supp}(\mu)=\{x\in\R\ |\ \mu(B_r(x))>0,\ \forall r>0\}.
$$
Following a standard atomisation strategy, we discretise the initial datum by splitting the total mass into equal parts as follows. Let $[\bar{x}_{min},\bar{x}_{max}]$ be the convex hull of the support of $\rho_0$ and let $[\bar{y}_{min},\bar{y}_{max}]$ be the convex hull of the support of  $\eta_0$. Fixing $N\in\mathbb{N}$ large enough, we split the non-negative subgraph of $\rho_0$ in $N$ regions of measure $\frac{1}{N}$ as follows:
\begin{subequations}\label{eq:discretisation_scheme}
\begin{align}
    &\bar{x}_0=\bar{x}_{min},\\
    &\bar{x}_i:=\sup\left\{x\in\R : \int_{\bar{x}_{i-1}}^xd\rho_0(y)<\frac{1}{N}\right\}, \qquad i=1,...,N.
\end{align}
\end{subequations}
Clearly, we have $\bar{x}_N=\bar{x}_{max}$. Repeating the same procedure for the non-negative subgraph of $\eta_0$ we obtain $\bar{y}_0,...,\bar{y}_N$. Then, we solve system \eqref{eq:ode-sys-newtonian}, for $i=1,\ldots,N$, with $(\bar{x}_1,...,\bar{x}_N,\bar{y}_1,...,\bar{y}_N)$ as initial condition. The choice of discarding the two particles labelled by $i=0$ is dictated by the need of having exactly $N$ particles each one with mass $1/N$.

In view of the results shown in Section \ref{sec:grad_flow}, we know there exists a unique solution $Z(t)=(x(t),y(t))\in \mathcal{C}^{N}\times \mathcal{C}^{N}$ in the sense of Definition \ref{def:discrete_gradflow} with support contained in the interval $[a(0),b(0)]$ for all times, with $a(0)=\min\{\bar{x}_0,\bar{y}_0\}$ and $b(0)=\max\{\bar{x}_N,\bar{y}_N\}$. Now, we consider the piecewise constant densities
\begin{align}\label{eq:piecewise-const-densities}
&\tilde{\rho}^N(t,x):=\sum_{i=0}^{N-1}d_i^1(t)\chi_{[x_i(t),x_{i+1}(t))}(x), \qquad \tilde{\eta}^N(t,x):=\sum_{i=0}^{N-1}d_i^2(t)\chi_{[y_i(t),y_{i+1}(t))}(x),
\end{align}
where
\begin{align}
\label{eq:discrete_densities}
    d_i^1(t)=\frac{1}{N(x_{i+1}(t)-x_i(t))}, \qquad \text{and}\qquad d_i^2(t)=\frac{1}{N(y_{i+1}(t)-y_i(t))}
\end{align}
are discrete Lagrangian version of the densities. Note that $d_i^1$ and $d_i^2$ are well-defined since particles of  the same species cannot collide, as proven in Section \ref{sec:properties}, Theorem \ref{thm:collisions_same_species}.  Moreover, we also consider the empirical measures 
\begin{equation}\label{eq:empirical_measures_sec_5}
\rho^N(t)=\frac{1}{N}\sum_{i=1}^{N}\delta_{x_i(t)}, \quad \text{and} \quad \eta^N(t)=\frac{1}{N}\sum_{j=1}^{N}\delta_{y_j(t)}.
\end{equation}
We notice that $\rho^N, \eta^N, \tilde{\rho}^N, \tilde{\eta}^N$ are probability measures with compact support. Moreover,  $(\tilde{\rho}^N(t),\tilde{\eta}^N(t))$ belong to $\mptra\times\mptra$.

Both $(\rho^N,\eta^N)$ and $(\tilde{\rho}^N,\tilde{\eta}^N)$ are useful representations of the particle system $x_1,\ldots,x_N,y_1,\ldots,y_N$ for large $N$. In fact, one can prove that these two sequences in $\mptr\times\mptr$ converge, up to a subsequence, to same limit in the $p$-Wasserstein distance for all $p\in [1,+\infty)$. 

\begin{lem}\label{lem:weak_measure_convergence}
Let $p\in [1,+\infty)$. There exists a sequence $(N_k)_k\subset \mathbb{N}$ and an absolutely continuous curve
\[(\rho(\cdot),\eta(\cdot))\in AC([0,T]\,;\,\,\mathcal{P}_p(\R)\times \mathcal{P}_p(\R))\]
such that 
\begin{align*}
    & (\rho^{N_k},\eta^{N_k})\rightarrow (\rho,\eta) \qquad \hbox{in}\ \  C([0,T]\,;\,\,\mathcal{P}_p(\R)\times \mathcal{P}_p(\R))\\
    & (\tilde{\rho}^{N_k},\tilde{\eta}^{N_k})\rightarrow (\rho,\eta) \qquad \hbox{in}\ \  C([0,T]\,;\,\,\mathcal{P}_p(\R)\times \mathcal{P}_p(\R))
\end{align*}
as $k\rightarrow +\infty$.
\end{lem}
\begin{proof}
Let us fix $T\geq 0$. The results in Proposition \ref{prop:support_solution} and  Proposition \ref{prop:uniform_subdiff} imply that 
\[
    \|Z(\cdot)\|_{L^\infty([0,T]\,;\,\R^N\times\R^N)}+\|\dot{Z}(\cdot)\|_{L^\infty([0,T]\,;\,\R^N\times\R^N)}\leq C,
\]
for some $C>0$ only depending on the initial support. The estimate on $\|Z(t)\|_\infty$ implies that all $q$-moments of $\rho^N$ and $\eta^N$ are uniformly bounded with respect to $N$, uniformly on $t\in [0,T]$, for $q\in [1, \infty)$. 

Therefore we may infer that both $\rho^N$ and $\eta^N$ are contained in a pre-compact subset of $\mathcal{P}_p(\R)$, for all times $t\in[0,T]$, by Prokhorov's theorem and the uniform bounds on the $q$-moments, with $q>p$. Now, for $0\leq s<t\leq T$, we set $\pi_{s,t}\in\mathcal{P}(\R\times\R)$ as
\[\pi_{s,t}^N(x,y):=\frac{1}{N}\sum_{i=1}^N \delta_{x_i(s)}(x)\otimes \delta_{x_i(t)}(y).\]
It is easily seen that $\pi_{s,t}^N$ has marginal measures $\rho^N(s)$ in the $x$-variable and $\rho^N(t)$ in the $y$-variable respectively. Hence, 
\begin{align*}
    \mW_p(\rho^N(s),\rho^N(t))^p & \leq \iint_{\R\times\R}|x-y|^p d\pi^N_{s,t}(x,y)=\frac{1}{N}\sum_{i=1}^N\iint_{\R\times\R}|x-y|^p d\delta_{x_i(s)}(x) d\delta_{x_i(t)}(y)\\
    & = \frac{1}{N}\sum_{i=1}^N |x_i(s)-x_i(t)|^p = \frac{1}{N}\sum_{i=1}^N\left| \int_s^t \dot{x}_i(\tau) d\tau \right|^p,
\end{align*}
and the above estimate on $\|\dot{Z}\|$ implies
\[\mW_p(\rho^N(s),\rho^N(t))^p\leq  \frac{C}{N}\sum_{i=1}^N |t-s|^p = C|t-s|^p,\]
for some constant $C>0$ that is independent of $N$. The latter estimate implies equi-continuity of the sequence $\{\rho^N\,:\,\,N\in \mathbb{N}\}$ in $C([0,T]\,;\,\mathcal{P}_p(\R))$, and clearly an analogous statement holds for $\eta^N$. Hence, the Arzelà-Ascoli's Theorem implies the existence of a subsequence $(\rho^{N_k},\eta^{N_k})$, $k\in \mathbb{N}$, such that 
\[
    (\rho^{N_k},\eta^{N_k})\rightarrow (\rho,\eta)\qquad \hbox{in $C([0,T]\,;\,\mathcal{P}_p(\R)\times\mathcal{P}_p(\R))$},
\]
as $k\rightarrow +\infty$, for some $(\rho,\eta)\in C([0,T]\,;\,\mathcal{P}_p(\R)\times\mathcal{P}_p(\R))$, see \cite[Proposition 3.3.1]{AGS}.

We now prove that the sequence $(\tilde{\rho}^{N_k},\tilde{\eta}^{N_k})$ converges to the same limit $(\rho,\eta)$ in the same topology $C([0,T]\,;\mP_p(\R)\times\mP_p(\R))$. For a fixed $N$, let $\pi^N\in C([0,T];\mathcal{P}(\R\times \R))$ be defined by
\[\pi^N(t;x,y) = \sum_{i=1}^N\delta_{x_i(t)}(x)\otimes \tilde{\rho}^N|_{[x_{i-1}(t),x_i(t))}(y).\]
A simple computation shows that $\pi^N$ has marginal measures $\rho^N$ in the $x$-variable and $\tilde{\rho}^N$ in the $y$-variable, respectively. Hence, for almost every $t\in [0,T]$,
\begin{align*}
    \mW_1(\rho^N(t),\tilde{\rho}^N(t)) &\leq \iint_{\R\times \R} |x-y|d\pi^N(x,y) = \sum_{i=1}^N \frac{1}{N(x_i(t)-x_{i-1}(t))}\int_{x_{i-1}(t)}^{x_i(t)}|x_i(t)-y|dy\\
    & = \frac{1}{2}\sum_{i=1}^N \frac{1}{N(x_i(t)-x_{i-1}(t))}(x_i(t)-x_{i-1}(t))^2 =(x_N(t)-x_0(t))\frac{1}{2N}\\
    &\leq (\bar{x}_N-\bar{x}_0)\frac{1}{2N}
\end{align*}
and the assertion is proven for $p=1$ by taking the supremum on $t\in [0,T]$ and using that $\bar{x}_N-\bar{x}_0=\bar{x}_{\max}-\bar{x}_{\min}$. Note that $\supp\,\rho^N,\supp\,\tilde{\rho}^N\subseteq[a(0),b(0)]$, hence
\begin{align*}
    \mW_p(\rho^N(t),\tilde{\rho}^N(t))&\leq \left(\iint_{[a(0),b(0)]^2} |x-y|^p d\pi^N(x,y)\right)^\frac{1}{p}\\
    &\leq\left(b(0)-a(0)\right)^{\frac{p-1}{p}}\iint_{\R\times \R} |x-y|d\pi^N(x,y)\\
    &\leq\left(b(0)-a(0)\right)^{\frac{p-1}{p}}(\bar{x}_N-\bar{x}_0)\frac{1}{2N},
\end{align*}
which gives the result for $p\in[1,+\infty)$ by taking again the supremum on $t\in[0,T]$ and letting $N\to \infty$.
\end{proof}

We now establish the basic properties satisfied by the $N$-particle approximation of the initial data $\rho_0,\eta_0$. In order to simplify the notation, we denote
\[(\tilde{\rho}^N,\tilde{\eta}^N)_{|_{t=0}} =(\tilde{\rho}_0^N,\tilde{\eta}_0^N)\,,\qquad (\rho^N,\eta^N)_{|_{t=0}} =(\rho^N_0,\eta^N_0).\]
\begin{prop}\label{prop:sequence_initial}
The two sequences $\{(\tilde{\rho}^N_0,\tilde{\eta}^N_0)\}_{n\in\mathbb{N}}$ and $\{(\rho^N_0,\eta^N_0)\}_{n\in\mathbb{N}}$ converge to the initial datum $(\rho_0,\eta_0)$ with respect to $\mW_1$. Moreover, assume that there exists a convex, non-decreasing function $G:[0,+\infty)\rightarrow [0,+\infty)$ with $G(0)=0$ and $\lim_{r\rightarrow +\infty}\frac{G(r)}{r}=+\infty$ such that both $G(\rho_0)$ and $G(\eta_0)$ belong to $L^1(\R)$. Then, the quantity
\[\int_\R G(\tilde{\rho}_0^N(x)) dx + \int_\R G(\tilde{\eta}_0^N(x)) dx \]
is uniformly bounded with respect to $N$.
\end{prop}

\begin{proof}
The $1$-Wasserstein convergence of the initial data relies on the techniques adopted in the proof of Lemma \ref{lem:weak_measure_convergence} and is therefore left to the reader. In order to prove the last property, we compute
\begin{align*}
    \int_\R G(\tilde{\rho}_0^N(x)) dx &= \sum_{i=0}^{N-1}\int_{\bar{x}_i}^{\bar{x}_{i+1}} G(d_i^1(0)) dx=\sum_{i=0}^{N-1}\int_{\bar{x}_i}^{\bar{x}_{i+1}} G\left(\frac{1}{N(\bar{x}_{i+1}-\bar{x}_i)}\right) dx \\
    &= \sum_{i=0}^{N-1}\int_{\bar{x}_i}^{\bar{x}_{i+1}} G\left(\avint_{\!\!\bar{x}_i}^{\bar{x}_{i+1}} \rho_0(y) dy \right) dx.
\end{align*}
By Jensen's inequality, we get
\[\int_\R G(\tilde{\rho}_0^N(x)) dx \leq \sum_{i=0}^{N-1}\int_{\bar{x}_i}^{\bar{x}_{i+1}}\frac{1}{\bar{x}_{i+1}-\bar{x}_i}\int_{\bar{x}_i}^{\bar{x}_{i+1}} G(\rho_0(y)) dy dx \leq \int_{\R}G(\rho_0(y)) dy,\]
which proves the assertion for $\tilde{\rho}_0^N$. The assertion for $\tilde{\eta}_0^N$ is proven in the same way.
\end{proof}

The convergence of $(\tilde{\rho}^N,\tilde{\eta}^N)$ and $(\rho^N,\eta^N)$ to $(\rho,\eta)$ proven in Lemma \ref{lem:weak_measure_convergence} alone is too weak to prove that $(\rho,\eta)$ is a gradient flow solution of the continuum model \eqref{eq:pde-sys-newtonian} in the sense of \cite{CDFEFS}. This is due to the discontinuity of the gradient $\nabla N$, which does not allow for the coupling with a singular measure in the mixed interaction terms of \eqref{eq:ode-sys-newtonian}. In order to bypass this problem, we argue as follows. Assuming the initial data $\rho_0,\eta_0$ belong to $L^m(\R)$ for $m\geq 1$, we aim at proving that the approximating sequences $\tilde{\rho}^N$ and $\tilde{\eta}^N$ are uniformly bounded in $L^m(\R)$, which implies weak $L^m$ compactness (note that the case $m=1$ is shown separately) and therefore the possibility to pass to the limit under the integral sign using the pairing between an absolutely continuous measure and a discontinuous test function.

\begin{prop}\label{prop:convergence_weak_Lm}
Let us consider $\rho_0,\eta_0\in\mptra\cap L^m(\R)$, for some  $m\in(1,+\infty]$. Then, the piecewise constant densities $\tilde{\rho}^N$ and $\tilde{\eta}^N$ have a weakly (resp. weakly star) convergent subsequence in $L_{loc}^m([0,+\infty)\times \R)$ for finite $m$ (infinite $m$ resp.) to $\rho$ and $\eta$ respectively. Moreover, $\rho$ and $\eta$ belong to $C([0,+\infty);\,L^m(\R))$.
\end{prop}
\begin{proof}
The proof is based on establishing $L^m$-bounds that are uniform in time and an application of Banach-Alaoglu theorem to obtain the weak-star compactness. Let us start by computing the time-derivative of the $L^m$ norm of the piecewise constant densities. In what follows, we use that particles of the same species do not cross, as proven in Theorem \ref{thm:collisions_same_species}. Moreover, the computation below is justified at times $t$ at which no particles of opposite species collide. As proven in Subsection \ref{subsec:collisions}, this only happens finitely many times on each fixed time interval $[0,T]$. Hence, for every fixed $T\geq0$ and for every but finitely many $t\in [0,T]$, we have
\begin{equation}
\label{eq:time-deriv-lm}
\begin{split}
    \frac{d}{dt} \left(\|\tilde \rho^N\|_m^m + \|\tilde \eta^N\|_m^m\right)
    &=\frac{d}{dt}\int_\R|\tilde{\rho}^N(t,x)|^m+|\tilde{\eta}^N(t,x)|^m\,dx\\
    &=-(m-1)\sum_{i=0}^{N-1}[d_i^1(t)]^m(\dot{x}_{i+1}(t)-\dot x_i(t))\\
    &\phantom{=\,} -(m-1)\sum_{j=0}^{N-1}[d_j^2(t)]^m(\dot{y}_{j+1}(t)-\dot y_j(t))\\
    &=-\frac{2(m-1)}{N}\sum_{i=0}^{N-1}[d_i^1(t)]^m(1-\alpha(i))\\
    &\phantom{=\,}-\frac{2(m-1)}{N}\sum_{j=0}^{N-1}[d_j^2(t)]^m(1-\beta(j))
\end{split}
\end{equation}
where $\alpha,\beta:\mathbb{N}\to\mathbb{N}$ are defined by \[\alpha(i)=\#\{k:x_i<y_k<x_{i+1}\},\qquad  \beta(j)=\#\{k:y_j<x_k<y_{j+1}\}\] 
for $i,j=\{0,...,N-1\}$. Clearly, the two maps $\alpha$ and $\beta$ also depend on time, but we will omit such dependency for simplicity and assume we are considering the above computation between two consecutive collisions.

Our goal is to show $\frac{d}{dt} \left(\|\tilde \rho^N\|_m^m + \|\tilde \eta^N\|_m^m\right)\le0$, which gives the desired $L^m$-bound uniform in time. From the last line of \eqref{eq:time-deriv-lm} this is true if $\alpha(i),\beta(j)\le1$ for all $i,j=\{0,...,N-1\}$. Now, let us rewrite \eqref{eq:time-deriv-lm} as follows
\begin{align*}
    \frac{d}{dt} \left(\|\tilde \rho^N\|_m^m + \|\tilde \eta^N\|_m^m\right)
    &=-\frac{2(m-1)}{N}\sum_{i:\alpha(i)=0}[d_i^1(t)]^m\\
    &\quad +\frac{2(m-1)}{N}\sum_{i:\alpha(i)>1}[d_i^1(t)]^m(\alpha(i)-1)\\
    &\quad -\frac{2(m-1)}{N}\sum_{j:\beta(j)=0}[d_j^2(t)]^m\\
    &\quad +\frac{2(m-1)}{N}\sum_{j:\beta(j)>1}[d_j^2(t)]^m(\beta(j)-1)\\
    &=: A_1 + A_2 + A_3 + A_4.
\end{align*}

We notice that, in case $\alpha(i)>1$, then there exist exactly  $\alpha(i)$ particles of the $y$-species, say with indices  $\bar{j},\bar{j}+1,\ldots,\bar{j}+\alpha(i)-1$, which are posed strictly between $x_i$ and $x_{i+1}$. For each $j\in \{\bar{j},\bar{j}+1,\ldots,\bar{j}+\alpha(i)-2\}$ we must have $\beta(j)=0$ since there are no $x$-particles between $y_j$ and $y_{j+1}$ for all the intermediate particles $y_j$ except the last one with $j=\alpha(i)-1$. Hence, the number $\alpha(i)-1$ equals exactly the number of $y$-particles between $x_i$ and $x_{i+1}$ characterised by $\beta(j)=0$, and 
$A_2$ can be re-written as
\begin{align*}
    A_2 &=\frac{2(m-1)}{N }\sum_{i:\alpha(i)>1} [d_i^1(t)]^m(\alpha(i)-1)\\
    &= \frac{2(m-1)}{N} \sum_{i:\alpha(i)>1} \sum_{\substack{j:\beta(j)=0\\x_i<y_j<x_{i+1}}} [d_i^1(t)]^m\\
    &\leq\frac{2(m-1)}{N} \sum_{i:\alpha(i)>1} \sum_{\substack{j:\beta(j)=0\\x_i<y_j<x_{i+1}}} [d_j^2(t)]^m,
\end{align*}
where the last inequality is motivated by the fact that for any index $i$ in the sum we have $d_i^1(t)\le d_j^2(t)$ because $y_{\bar{j}+k}-y_{\bar{j}+k-1} \leq x_{i+1}-x_i$ for all $k\in \{1,\ldots,\alpha(i)-2\}$. Now, we claim that 
\begin{equation}\label{eq:claim}
   \sum_{i:\alpha(i)>1} \sum_{\substack{j:\beta(j)=0\\x_i<y_j<x_{i+1}}}[d_j^2(t)]^m = \sum_{j:\beta(j)=0}[d_j^2(t)]^m. 
\end{equation}
Indeed, the set of indexes $\{j\,:\,\beta(j)=0\}$ can be split into a finite number $k$ of sets $I_1,\ldots, I_k$, with $I_i\cap I_j=\emptyset$ if $i\neq j$, with each $I_\ell$ made up by $h_\ell$ consecutive elements, say of the form $I_\ell=\{\bar{j},\ldots,\bar{j}+h_\ell-1\}$. Without restriction, we can assume that the sets $I_\ell$ are maximal with respect to those properties, i. e. no union of any such $I_i\cup I_j$ with $i\neq j$ is made up by consecutive indexes. In such configuration, for each $\ell\in \{1,\ldots,k\}$ we can detect a unique $i\in \{1,\ldots,N\}$ such that $x_i<y_j<x_{i+1}$ for all $j\in \{\bar{j},\ldots,\bar{j}+h_\ell-1\}$, and this implies $\ell=\alpha(i)$, which proves our previous claim \eqref{eq:claim}. As a consequence of \eqref{eq:claim}, we immediately get $A_2+A_3 \leq 0$. Arguing in a similar way we also get $A_1 + A_4 \leq 0$, which gives $\frac{d}{dt}\left(\|\tilde \rho^N\|_m^m + \|\tilde \eta^N\|_m^m\right) \leq 0$ on each time interval between two consecutive collisions, whence
$$
    \|\tilde{\rho}^N\|_{L^\infty([0,T]; L^m(\R))} + \|\tilde \eta^N\|_{L^\infty([0,T]; L^m(\R))} \leq \|\tilde{\rho}^N(\cdot,0)\|_{L^m(\R)} + \|\tilde{\eta}^N(\cdot,0)\|_{L^m(\R)}.
$$
The last estimate can be extended to the case $m=+\infty$ as we observe
\begin{align*}
     & \|\tilde{\rho}^N(\cdot,t)\|_{L^\infty(\R)}+\|\tilde{\eta}^N(\cdot,t)\|_{L^\infty(\R)}\leq \limsup_{m\rightarrow +\infty}\left[\|\tilde{\rho}^N(\cdot,t)\|_{L^m(\R)}+\|\tilde{\eta}^N(\cdot,t)\|_{L^m(\R)}\right]\\
      &\  \leq \limsup_{m\rightarrow +\infty}\left[\|\tilde{\rho}^N(\cdot,0)\|_{L^m(\R)} + \|\tilde{\eta}^N(\cdot,0)\|_{L^m(\R)}\right]\\
      &\ \leq \limsup_{m\rightarrow +\infty}\left[\|\tilde{\rho}^N(\cdot,0)\|_{L^\infty(\R)}^{\frac{m-1}{m}}\|\tilde{\rho}^N(\cdot,0)\|_{L^1(\R)}^{\frac{1}{m}}+\|\tilde{\eta}^N(\cdot,0)\|_{L^\infty(\R)}^{\frac{m-1}{m}}
      \|\tilde{\eta}^N(\cdot,0)\|_{L^1(\R)}^{\frac{1}{m}}\right]\\
      &\  = \|\tilde{\rho}^N(\cdot,0)\|_{L^\infty(\R)} + \|\tilde{\eta}^N(\cdot,0)\|_{L^\infty(\R)}.
\end{align*}
Therefore, due to Proposition \ref{prop:sequence_initial} with $G(r)=r^m$, the sequences $\{\tilde{\rho}^N\}_{N\in\mathbb{N}}$ and $\{\tilde{\eta}^N\}_{N\in\mathbb{N}}$ are uniformly bounded in $L_{loc}^\infty([0,+\infty);L^m(\R))$. By weak compactness, if $m<+\infty$ there exists a subsequence for each of them converging weakly in $L^m_{loc}([0,\infty)\times\R)$ to some limits $\rho',\eta'\in L_{loc}^m([0,+\infty)\times \R)$, respectively. In the case of $m=+\infty$ the above subsequence converges in the weak-$\star$ topology of $L_{loc}^\infty([0,+\infty)\times \R)$. In view of Lemma \ref{lem:weak_measure_convergence}, the limits $\rho'$ and $\eta'$ coincide with $\rho$ and $\eta$ respectively. The last statement follows by weak lower semi-continuity of the $L^m$ norm.
\end{proof}

The above weak compactness can be stretched to the $m=1$ case.
\begin{prop}\label{cor:dunford}
Let us consider $\rho_0,\eta_0 \in \mptr\cap L^1(\R)$. Then $\tilde{\rho}^N$ and $\tilde{\eta}^N$ converge weakly (up to a subsequence) in $L^1_{loc}([0,+\infty)\times\R)$ to $\rho$ and $\eta$ respectively. Consequently, $\rho$ and $\eta$ belong to $L^\infty([0,+\infty);\,L^1(\R))$.
\end{prop}

\begin{proof}
By de la Vallée-Poussin's Theorem, there exists a non-decreasing, convex function $G:[0,+\infty)\rightarrow [0,+\infty)$ with $G(0)=0$ and $\lim_{r\rightarrow +\infty}\frac{G(r)}{r}=+\infty$ such that $G(\rho_0), G(\eta_0) \in L^1(\R)$. Hence, Proposition \ref{prop:sequence_initial} implies that both $G(\tilde{\rho}^N_0)$ and $G(\tilde{\eta}^N_0)$ are uniformly bounded in $L^1(\R)$. By repeating the proof of Proposition \ref{prop:convergence_weak_Lm} with $G(d_i^j)$ instead of $(d_i^j)^m$ with $j=1,2$ and $i=0,\ldots,N-1$, we easily get a uniform bound for
\[\|G(\tilde{\rho}^N)\|_{L^\infty(([0,+\infty); L^1(\R))} + \|G(\tilde \eta^N)\|_{L^\infty(([0,+\infty); L^1(\R))}.\]
In fact, by using the same notation of Proposition \ref{prop:convergence_weak_Lm} we get
\begin{align*}
    \frac{d}{dt}\int_\R G(\tilde\rho^N(t))+ G(\tilde\eta^N(t))\,dx&=\frac{d}{dt}\sum_{i=0}^{N-1}G(d_i^1(t))(x_{i+1}(t)-x_i(t))+\frac{d}{dt}\sum_{j=0}^{N-1}G(d_j^2(t))(y_{j+1}(t)-y_j(t))\\
    &=-\frac{2}{N}\sum_{i=0}^{N-1}G'(d_i^1(t))d_i^1(t)(1-\alpha(i)) + \frac{2}{N}\sum_{i=0}^{N-1}G(d_i^1(t))(1-\alpha(i))\\
    &\quad-\frac{2}{N}\sum_{j=0}^{N-1}G'(d_j^2(t))d_j^2(t)(1-\beta(j)) + \frac{2}{N}\sum_{j=0}^{N-1}G(d_j^2(t))(1-\beta(j))\\
    &=-\frac{2}{N}\sum_{i=0}^{N-1}[G'(d_i^1(t))d_i^1(t)-G(d_i^1(t))](1-\alpha(i))\\
    &\quad - \frac{2}{N}\sum_{j=0}^{N-1}[G'(d_j^2(t))d_j^2(t)-G(d_j^2(t))](1-\beta(j)).
\end{align*}
As mentioned above we can argue as in the proof of Proposition \ref{prop:convergence_weak_Lm} since $G$ is convex, hence the function $x\in(0,+\infty)\mapsto G'(x)x-G(x)$ is non-decreasing.
Therefore, by the de la Vall\'ee-Poussin's theorem, we may infer the equi-integrability of the sequences $\tilde{\rho}^N$ and $\tilde{\eta}^N$, and thus, by an application of the Dunford-Pettis theorem the two sequences are weakly compact in $L^1_{loc}([0,+\infty)\times \R)$. Hence, Lemma \ref{lem:weak_measure_convergence} implies that the limits $\rho'$ and $\eta'$ coincide with $\rho$ and $\eta$ respectively. The last statement follows by weak lower semi-continuity of the $L^1$ norm.
\end{proof}

The following technical lemma will be useful in the proof of our main result.
\begin{lem}\label{lem:technical}
For all $N\in \mathbb{N}$, let
\[\tilde{F}^N(x,t)=\int_{-\infty}^x \tilde{\rho}^N(y,t) dy\,,\qquad \tilde{H}^N(x,t)=\int_{-\infty}^x \tilde{\eta}^N(y,t) dy.\]
Then, the two families $\{\tilde{F}^N\}_{N\in \mathbb{N}}$ and $\{\tilde{H}^N\}_{N\in \mathbb{N}}$ are strongly compact in $L^1_{loc}(\R\times [0,+\infty))$.
\end{lem}

\begin{proof}
Since both $\tilde{\rho}^N(\cdot,t)$ and $\tilde{\eta}^N(\cdot,t)$ have unit mass for all $t\geq 0$, we immediately get
\begin{equation}\label{eq:estimate1}
  \sup_{t\geq 0}\left[\|\tilde{F}^N(\cdot,t)\|_{L^\infty(\R)} + \|\tilde{H}^N(\cdot,t)\|_{L^\infty(\R)}\right] <+\infty.  
\end{equation}
Moreover, from the proof of Proposition \ref{cor:dunford} we easily obtain 
\begin{equation}\label{eq:estimate2}
    \sup_{t\geq 0}\left[\|G(\tilde{F}^N_x(t,\cdot))\|_{L^1(\R)} + \|G(\tilde{H}^N_x(t,\cdot))\|_{L^1(\R)}\right]<+\infty,
\end{equation}
where $G$ is a function as in the statement of Proposition \ref{prop:sequence_initial}, the existence of which is guaranteed by de la Vallée-Poussin's Theorem. Now, in order to estimate the oscillations in time, we aim at proving some uniform equi-continuity in time of the curve $t\mapsto (\tilde{\rho}^N(\cdot,t),\tilde{\eta}^N(\cdot,t))$ in the $1$-Wasserstein distance. To perform this task, for $0\leq s<t$ we recall, from the the content of Section \ref{sec:preliminaries}, that
\[\mW_1(\tilde{\rho}^N(t),\tilde{\rho}^N(s)) = \|\tilde{X}^N(\cdot,t)-\tilde{X}^N(\cdot,s)\|_{L^1([0,1])},\]
where $\tilde{X}^N:[0,1]\times [0,+\infty)\rightarrow \R$ is the pseudo-inverse with respect to the $x$-variable of cumulative distribution $\tilde{F}^N$ defined above. A simple computation yields
\begin{align*}
    \tilde{X}^N(z,t)=X_{\tilde{\rho}^N}(z,t)= &
        \sum_{i=0}^{N-2}\left[x_i(t)+\frac{1}{d_i^1(t)}\left(z-\frac{i}{N}\right)\right]\chi_{[\frac{i}{N},\frac{i+1}{N})}(z)\\
        &+\left[x_{N-1}(t)+\frac{1}{d_{N-1}^1(t)}\left(z-\frac{N-1}{N}\right)\right]\chi_{[\frac{N-1}{N},1]}(z).
\end{align*}
Hence,
\begin{align*}
    & \|\tilde{X}^N(\cdot,t)-\tilde{X}^N(\cdot,s)\|_{L^1([0,1])} \\
    & \ \leq \sum_{i=0}^{N-1}\int_{i/N}^{(i+1)/N}\left[|x_i(t)-x_i(s)|+ N \left(|x_{i+1}(t)-x_{i+1}(s)| + |x_i(t)-x_i(s)|\right)\left(z-\frac{i}{N}\right)\right] dz.
\end{align*}
Similarly to the proof of Lemma \ref{lem:weak_measure_convergence}, Proposition \ref{prop:uniform_subdiff} implies there exists a constant $C\geq 0$ independent of $N$ such that
\begin{align*}
    &  \|\tilde{X}^N(\cdot,t)-\tilde{X}^N(\cdot,s)\|_{L^1([0,1])}  \leq \frac{C}{N} \sum_{i=0}^{N-1}|t-s| = C|t-s|.
\end{align*}
Consequently, we obtain
\begin{equation}\label{eq:estimate3}
    \|\tilde{F}^N(\cdot,t)-\tilde{F}^N(\cdot,s)\|_{L^1(\R)}\leq C|t-s|,
\end{equation}
and a similar estimate can be also deduced for $\tilde{H}^N(\cdot,t)$. Combining estimates \eqref{eq:estimate1}, \eqref{eq:estimate2}, and \eqref{eq:estimate3}, for every compact subset $K\subset \R$ we obtain that $\tilde{F}^N$ is an equi-continuous family of absolutely continuous curves with values on a bounded and compact subset of $L^1(K)$, where we are also using Dunford-Pettis Theorem. By Arzelà-Ascoli Theorem, $\tilde{\rho}^N$ is strongly compact in $L^1([0,T]\times K)$ and the same holds for $\tilde{\eta}^N$, which proves the assertion.
\end{proof}

We are now ready to prove the main result of this section.
\begin{thm}\label{thm:convergence_main_1}
Let $m\in[1,+\infty]$ and $(\rho_0,\eta_0)\in(\mptra\cap L^m(\R))^2$ with compact support. Then, the piecewise constant particle approximation $(\tilde{\rho}^N,\tilde{\eta}^N)$ converges, up to a subsequence, weakly in $L^m_{loc}([0,+\infty)\times\R)^2$ to the unique weak measure solution $(\rho,\eta)$ to system \eqref{eq:pde-sys-newtonian} according to Definition \ref{def:solutions} with initial datum $(\rho_0,\eta_0)$. The empirical measure approximation $(\rho^N,\eta^N)$ converges, up to a subsequence, towards the same limit in $C([0,+\infty)\,;\,\mathcal{P}_p(\R)^2)$.
\end{thm}

\begin{proof}
Our goal is to show that the limit pair $(\rho,\eta)$ satisfies \eqref{eq:weak_formulation}. We shall prove the statement for the first equation in \eqref{eq:weak_formulation}, the second one being done in the same vein. We start by proving that the approximating measure $(\rho^N,\eta^N)$ almost satisfies the first equation in \eqref{eq:weak_formulation}, up to removing the diagonal $x=y$ to avoid the discontinuity of the $\sign$-function. Let $T\geq 0$ be a fixed time and let $\varphi\in C_c^1([0,T)\times \R)$. We have:
\begin{align}
\label{eq:weak_formulation_particles}
    \begin{split}
    & \int_0^T \int_\R \varphi_t(x,t) d\rho^N(t)(x) dt +\int_\R \varphi(x,0)d\rho_0^N(x)\\
    & + \int_0^T\iint_{\R\times\R\setminus\{x=y\}} \varphi_x(x,t)\sign(x-y)d\rho^N(t)(y)d\rho^N(t)(x)dt\\
    &  -  \int_0^T\iint_{\R\times\R\setminus\{x=y\}} \varphi_x(x,t)\sign(x-y)d\eta^N(t)(y)d\rho^N(t)(x)dt\\
    & \ = \frac{1}{N}\sum_{i=1}^N\int_0^T\varphi_t(x_i(t),t) dt + \frac{1}{N}\sum_{i=1}^N\varphi(\bar{x}_i,0)\\
    &\quad  \ + \frac{1}{N^2}\int_0^T \sum_{i=1}^N \sum\limits_{\substack{j=1 \\ x_i\neq x_j }}^N \sign(x_i(t)-x_j(t)) \varphi_x(x_i(t),t)dt\\
    &  \quad \ - \frac{1}{N^2}\int_0^T \sum_{i=1}^N \sum\limits_{\substack{j=1 \\ x_i\neq y_j }}^N \sign(x_i(t)-y_j(t)) \varphi_x(x_i(t),t)dt.
    \end{split}
\end{align}
Applying the chain rule in the first term and the assumption on the support of $\varphi$ imply
\begin{align*}
    &  \frac{1}{N}\sum_{i=1}^N\int_0^T\varphi_t(x_i(t),t) dt + \frac{1}{N}\sum_{i=1}^N\varphi(\bar{x}_i,0)= -\frac{1}{N}\sum_{i=1}^N\int_0^T \dot{x}_i(t)\varphi_x(x_i(t),t)dt.
\end{align*}
We remind the reader that particles of the same species never collide and we observe that, as a consequence of Subsection \ref{subsec:collisions}, only a finite number of collisions between particles of the two species occurs in the time interval $[0,T)$. Hence, since the particles $x_i$ satisfy \eqref{eq:ode-sys-newtonian} away from the collision times, the right hand side of \eqref{eq:weak_formulation_particles} equals zero. In order to conclude the proof, we need to show that the left-hand side of \eqref{eq:weak_formulation_particles} tends to 
\begin{align*}
     & \int_0^T \int_\R \varphi_t(x,t) \rho(x,t) dx dt +\int_\R \varphi(x,0)\rho_0(x) dx\\
    & \ + \int_0^T\iint_{\R\times\R} \varphi_x(x,t)\sign(x-y)\rho(y,t)\rho(x,t)dy dx dt\\
    & \ -  \int_0^T\iint_{\R\times\R} \varphi_x(x,t)\sign(x-y)\eta(y,t)\rho(x,t)dy dx dt,
\end{align*}
as $N\rightarrow +\infty$. The proof would be completed in this case as  $\rho(\cdot,t)$ and $\eta(\cdot,t)$ being in $L^1(\R)$ at each time will make sure the diagonal terms in the above integrals do not bring any contribution. 

First, the weak measure convergence of $\rho^N$ to $\rho$ and of $\rho^N_0$ to $\rho_0$ easily implies
\begin{align*}
     & \int_0^T \int_\R \varphi_t(x,t) d\rho^N(t)(x) dt +\int_\R \varphi(x,0)d\rho_0^N(x)\\
     & \ \rightarrow \int_0^T \int_\R \varphi_t(x,t) \rho(x,t) dx dt +\int_\R \varphi(x,0)\rho_0(x)dx.
\end{align*}
Hence, in order to conclude we only need to prove that in the $N\rightarrow +\infty$ limit we have
\begin{align*}
    & \frac{1}{N^2}\sum_{i=1}^N \sum\limits_{\substack{j=1 \\ x_i\neq x_j }}^N \int_0^T\sign(x_i(t)-x_j(t)) \varphi_x(x_i(t),t)dt- \frac{1}{N^2}\sum_{i=1}^N \sum\limits_{\substack{j=1 \\ x_i\neq y_j }}^N \int_0^T\sign(x_i(t)-y_j(t)) \varphi_x(x_i(t),t)dt \\
    & \longrightarrow \int_0^T\iint_{\R\times\R}\varphi_x(x,t)\rho(x,t)\rho(y,t)\sign(x-y)dx dy dt \\
    & \qquad \qquad - \int_0^T\iint_{\R\times\R} \varphi_x(x,t)\sign(x-y)\eta(y,t)\rho(x,t)dy dx dt.
\end{align*}
The following holds:
\begin{align}
    \label{eq:RHS_weak_app}
    \begin{split}
    \phantom{=}&\frac{1}{N^2}\sum_{i=1}^N \sum\limits_{\substack{j=1 \\ x_i\neq x_j }}^N \int_0^T\sign(x_i(t)-x_j(t)) \varphi_x(x_i(t),t)dt - \frac{1}{N^2}\sum_{i=1}^N \sum\limits_{\substack{j=1 \\ x_i\neq y_j }}^N \int_{0}^{T}\sign(x_i(t)-y_j(t)) \varphi_x(x_i(t),t)dt \\
    &=\frac1N\sum_{i=1}^N \int_0^T \varphi_x(x_i(t), t) \left\{\frac1N \sum\limits_{\substack{j=1 \\ x_i\neq x_j }}^N \sign(x_i(t)-x_j(t)) -\frac1N \sum\limits_{\substack{j=1 \\ x_i\neq y_j }}^N \sign(x_i(t)-y_j(t)) \right\}dt.
    \end{split}
\end{align}
Let us focus on the terms in the parentheses. We have
\begin{align*}
    \frac1N &\sum\limits_{\substack{j=1 \\ x_i\neq x_j} }^N \sign(x_i(t)-x_j(t)) - \frac1N \sum\limits_{\substack{j=1 \\ x_i\neq y_j }}^N\sign(x_i(t)-y_j(t))\\
    &= \rho^N((-\infty, x_i(t))) - \rho^N((x_i(t), \infty)) - \eta^N((-\infty, x_i(t)))  + \eta^N((x_i(t), \infty)).
\end{align*}
It is now an easy consequence of the definition of the cumulative distribution functions,
\[
    F^N(x,t)=\rho^N((-\infty,x])\quad \mbox{and}\quad H^N(x,t)=\eta^N((-\infty,x]),
\]
that 
\begin{align}
    \label{eq:parentsterms}
    \begin{split}
     \frac1N &\sum\limits_{\substack{j=1 \\ x_i\neq x_j} }^N \sign(x_i(t)-x_j(t)) -\frac1N \sum\limits_{\substack{j=1 \\ x_i\neq y_j }}^N \sign(x_i(t)-y_j(t))\\
    &= \rho^N((-\infty, x_i(t))) - \rho^N((x_i(t), \infty)) - \eta^N((-\infty, x_i(t)))  + \eta^N((x_i(t), \infty))\\
     &= 2 F^N(x_i(t)) - 1 - \rho^N(\{x_i(t)\}) - (2 H^N(x_i(t)) - 1 - \eta^N(\{x_i(t)\}))\\
     &= 2(F^N(x_i(t)) - H^N(x_i(t))) - \rho^N(\{x_i(t)\}) + \eta^N(\{x_i(t)\})\\
     &= 2(F^N(x_i(t)) - H^N(x_i(t))) - 1/N + \eta^N(\{x_i(t)\}).
     \end{split}
\end{align}
Substituting Eq. \eqref{eq:parentsterms} into Eq. \eqref{eq:RHS_weak_app}, we obtain
\begin{align*}
    &\frac{1}{N^2}\sum_{i=1}^N \sum\limits_{\substack{j=1 \\ x_i\neq x_j }}^N \int_0^T\sign(x_i(t)-x_j(t)) \varphi_x(x_i(t),t)dt - \frac{1}{N^2}\sum_{i=1}^N \sum\limits_{\substack{j=1 \\ x_i\neq y_j }}^N \int_{0}^{T}\sign(x_i(t)-y_j(t)) \varphi_x(x_i(t),t)dt \\
    & \ = \frac{2}{N}\sum_{i=1}^N \int_{0}^{T} \varphi_x(x_i(t),t) \left[F^N(x_i(t),t)- H^N(x_i(t),t)\right] dt\\
    &\qquad +\frac{1}{N}\sum_{i=1}^N\int_{0}^{T} \varphi_x(x_i(t),t) \left( \eta^N(\{x_i(t)\})- \frac1N \right)\,dt\\
    & \ = \frac{2}{N}\sum_{i=1}^N \int_{0}^{T} \varphi_x(x_i(t),t) \left[F^N(x_i(t),t)- H^N(x_i(t),t)\right] dt + \mathcal{O}(1/N),
\end{align*}
since $\eta^N\left(\{x_i(t)\}\right)$ can only be either $0$ or $1/N$, and in the former case it holds
\begin{align*}
\left|-\frac{1}{N^2}\sum_{i=1}^N\int_{0}^{T} \varphi_x(x_i(t),t)\,dt\right|\le\frac{T}{N}\|\varphi\|_{L^\infty}.
\end{align*}
Now, denoting $\tilde{F}^N$ and $\tilde{H}^N$ as in Lemma \ref{lem:technical}, we get
\begin{align*}
    & \frac{2}{N}\sum_{i=1}^N \int_{0}^{T} \varphi_x(x_i(t),t) F^N(x_i(t),t) dt \\
    & \ = \frac{2}{N}\sum_{i=1}^N \int_{0}^{T} \varphi_x(x_i(t),t) \left(F^N(x_i(t),t)-\tilde{F}^N(x_i(t),t)\right) dt + \frac{2}{N}\sum_{i=1}^N \int_{0}^{T} \varphi_x(x_i(t),t) \tilde{F}^N(x_i(t),t) dt\\
    & \ =\frac{2}{N}\sum_{i=1}^N \int_{0}^{T} \varphi_x(x_i(t),t) \tilde{F}^N(x_i(t),t) dt,
\end{align*}
and
\begin{align*}
    & \frac{2}{N}\sum_{i=1}^N \int_{0}^{T} \varphi_x(x_i(t),t) H^N(x_i(t),t) dt \\
    & \ = \frac{2}{N}\sum_{i=1}^N \int_{0}^{T} \varphi_x(x_i(t),t) \left(H^N(x_i(t),t)-\tilde{H}^N(x_i(t),t)\right) dt + \frac{2}{N}\sum_{i=1}^N \int_{0}^{T} \varphi_x(x_i(t),t) \tilde{H}^N(x_i(t),t) dt.
\end{align*}
Since
\begin{align*}
    & \left|\frac{2}{N}\sum_{i=1}^N \int_{0}^{T} \varphi_x(x_i(t),t) \left(H^N(x_i(t),t)-\tilde{H}^N(x_i(t),t)\right) dt\right|\\
    & \ \leq \|\varphi_x\|_{L^\infty}\frac{2}{N}\sum_{i=1}^N \int_{0}^{T} \left|(H^N(x_i(t),t)-\tilde{H}^N(x_i(t),t)\right| dt\leq \frac{C(T)}{N}\|\varphi_x\|_{L^\infty} ,
\end{align*}
we easily obtain
\begin{align*}
    & \frac{1}{N^2}\sum_{i=1}^N \sum_{j=1}^N \int_0^T\sign(x_i(t)-x_j(t)) \varphi_x(x_i(t),t)dt - \frac{1}{N^2}\sum_{i=1}^N \sum\limits_{\substack{j=1 \\ x_i\neq y_j }}^N \int_0^T\sign(x_i(t)-y_j(t)) \varphi_x(x_i(t),t)dt \\
    & \ = \frac{2}{N}\sum_{i=1}^N \int_{0}^{T} \varphi_x(x_i(t),t) \tilde{F}^N(x_i(t),t) dt -\frac{2}{N}\sum_{i=1}^N \int_{0}^{T} \varphi_x(x_i(t),t) \tilde{H}^N(x_i(t),t) dt + O(1/N),
\end{align*}
as $N\rightarrow +\infty$. We now compute
\begin{align*}
    & \frac{2}{N}\sum_{i=1}^N \int_{0}^{T} \varphi_x(x_i(t),t) \tilde{F}^N(x_i(t),t) dt-\frac{2}{N}\sum_{i=1}^N \int_{0}^{T} \varphi_x(x_i(t),t) \tilde{H}^N(x_i(t),t) dt \\
    & \ = 2\sum_{i=1}^N \int_{0}^{T} \int_{x_{i-1}(t)}^{x_{i}(t)}\frac{1}{N(x_i(t)-x_{i-1}(t))}\varphi_x(x_i(t),t) \tilde{F}^N(x_i(t),t) \,dx\, dt\\ 
    & \quad -2\sum_{i=1}^N \int_{0}^{T} \int_{x_{i-1}(t)}^{x_{i}(t)}\frac{1}{N(x_i(t)-x_{i-1}(t))}\varphi_x(x_i(t),t) \tilde{H}^N(x_i(t),t) \,dx\, dt\\
    & \ = 2\int_{0}^{T}\int_\R \tilde{\rho}^N(x,t)\varphi_x(x,t)\tilde{F}^N(x,t) \,dx\, dt -2\int_{0}^{T}\int_\R \tilde{\rho}^N(x,t)\varphi_x(x,t)\tilde{H}^N(x,t) \,dx\, dt+ R(N,T),
\end{align*}
with
\begin{align*}
    & |R(N,T)|\leq  2\sum_{i=1}^N\int_{0}^{T}\int_{x_{i-1}(t)}^{x_i(t)}\frac{1}{N(x_i(t)-x_{i-1}(t))}\left|\varphi_x(x_i(t),t)-\varphi_x(x,t)\right|\tilde{F}^N(x_i(t),t)dx dt\\
    & \ + 2\sum_{i=1}^N\int_{0}^{T}\int_{x_{i-1}(t)}^{x_i(t)}\frac{1}{N(x_i(t)-x_{i-1}(t))}|\varphi_x(x,t)|\left|\tilde{F}^N(x_i(t),t)-\tilde{F}^N(x,t)\right| dx dt\\  
    & \ + 2\sum_{i=1}^N\int_{0}^{T}\int_{x_{i-1}(t)}^{x_i(t)}\frac{1}{N(x_i(t)-x_{i-1}(t))}\left|\varphi_x(x_i(t),t)-\varphi_x(x,t)\right|\tilde{H}^N(x_i(t),t)dx dt\\
    & \ +2\sum_{i=1}^N\int_{0}^{T}\int_{x_{i-1}(t)}^{x_i(t)}\frac{1}{N(x_i(t)-x_{i-1}(t))}|\varphi_x(x,t)|\left|\tilde{H}^N(x_i(t),t)-\tilde{H}^N(x,t)\right| dx dt\\
    & \ \leq \frac{2}{N}\|\varphi_{xx}\|_{L^\infty}\sum_{i=1}^N\int_0^T (x_i(t)-x_{i-1}(t)) dt + \frac{4T}{N}\|\varphi_x\|_{L^\infty}\leq \frac{C}{N},
\end{align*}
for some constant $C\geq 0$ depending on the support of the initial datum $\rho_0$, on $T$, and on the test function $\varphi$. Combining the above estimates we obtain
\begin{equation}\label{eq:pass_to_the_limit}
\begin{split}
    & \frac{1}{N^2}\sum_{i=1}^N \sum_{j=1}^N \int_0^T\sign(x_i(t)-x_j(t)) \varphi_x(x_i(t),t)dt - \frac{1}{N^2}\sum_{i=1}^N \sum\limits_{\substack{j=1 \\ x_i\neq y_j }}^N \int_0^T\sign(x_i(t)-y_j(t)) \varphi_x(x_i(t),t)dt \\
    & \ = 2\int_{0}^{T}\int_\R \tilde{\rho}^N(x,t)\varphi_x(x,t)\tilde{F}^N(x,t) \,dx\, dt -2\int_{0}^{T}\int_\R \tilde{\rho}^N(x,t)\varphi_x(x,t)\tilde{H}^N(x,t) \,dx\, dt + O(1/N)\\
    & \ =\int_0^T \iint_{\R\times\R}\sign(x-y)\varphi_x(x,t)\tilde{\rho}^N(y,t)\tilde{\rho}^N(x,
    t)\, dy\, dx\, dt \\
    & \quad -\int_0^T \iint_{\R\times\R}\sign(x-y)\varphi_x(x,t)\tilde{\eta}^N(y,t)\tilde{\rho}^N(x,
    t)\, dy\, dx\, dt+ O(1/N)
    \end{split}
\end{equation}
as $N\rightarrow +\infty$. Now, since $\tilde{\rho}^N$ and $\tilde{\eta}^N$ are weakly compact in $L^1_{loc}([0,T]\times \R)$ according to Proposition \ref{cor:dunford}, then so are the product measure $\tilde{\rho}^N(\cdot,t)\otimes \tilde{\rho}^N(\cdot,t)$ and $\tilde{\rho}^N(\cdot,t)\otimes \tilde{\eta}^N(\cdot,t)$ on $[0,T]\times\R\times\R$. Hence, we can pass to the limit in the last term of \eqref{eq:pass_to_the_limit} and obtain the desired assertion. It is straightforward to extend the result to an $L_{loc}^m$-setting since we can readily apply Proposition \ref{prop:convergence_weak_Lm}  to infer weak $L^m$-compactness.
\end{proof}

\begin{rem}\label{rem:garroni}
As pointed out in the introduction, our results allows to establish a rigorous link between a discrete model such as \eqref{eq:ode-sys-newtonian} and the continuum system of PDEs \eqref{eq:pde-sys-newtonian}. A similar result is proven in \cite{garroni20} for a general interaction kernel, possibly with logarithmic repulsive singularity, by regularising the interaction potential $V$ in the discrete setting by $V_{\delta_n}$ having second derivative bounded in $L^\infty$ by $\lambda_n:=\|D
^2V_{\delta_n}\|_{L^\infty}$. However, the result requires, see Theorem 3.3 and Remark 3.4 of \cite{garroni20}, for a general initial condition in $L^1\log L^1$, that
\[e^{3T\lambda_{\delta_N}} N^{-1}\rightarrow 0\]
an $N\rightarrow +\infty$. By smoothing our interaction potential $V(x)=-|x|$ on an interval $[-\delta_N,\delta_N]$ we obtain the necessary condition that $\delta_N$ must be tending to zero slower than $\frac{3T}{\log N}$ as $N\rightarrow +\infty$. For a simple initial condition $\rho_0(x)=\mathbf{1}_{[0,1]}$ this implies that a considerable portion of interactions are artificially \say{damped} in the discrete model. Indeed, since any two consecutive particles have a distance  of order $1/N$ in the case of the above initial condition, the regularisation by $V_{\delta_n}$ impacts on the interaction of each particle with a number of particles of order $\frac{N}{\log N}$. Our approach on the other hand allows, in the one-dimensional case and with $V(x)=-|x|$, to avoid any regularisation in the discrete setting.
\end{rem}

\section*{Acknowledgments}
A considerable part of this work was carried out during the visit of MDF to King Abdullah University of Science and Technology (KAUST) in Thuwal, Saudi Arabia. MDF is deeply grateful for the warm hospitality by people at KAUST,  for the excellent scientific environment, and for the support in the development of this work. AE was partially supported by the German Science Foundation (DFG) through CRC TR 154  "Mathematical Modelling, Simulation and Optimization Using the Example of Gas Networks". AE and MS gratefully acknowledge the support of the Hausdorff Research Institute for Mathematics (Bonn), through the Junior Trimester Program on \emph{Kinetic Theory}.

\bibliography{references}
\bibliographystyle{plain}

\end{document}